\documentclass[11pt]{amsart}
\pdfoutput=1
\newcommand{\ds}{\displaystyle}
\newcommand{\ben}{\begin{enumerate}}
\newcommand{\een}{\end{enumerate}}
\newcommand{\eq}[2][label]{\begin{equation}\label{#1}#2\end{equation}}
\newcommand{\av}[2]{\langle #1\rangle_{_{\scriptstyle #2}}}
\newcommand{\ave}[1]{\langle #1\rangle}
\newcommand{\ve}{\varepsilon}
\usepackage{amsfonts}
\usepackage{amssymb}
\usepackage{amsmath}
\usepackage{graphicx,color}

\newcommand{\bel}[1]{\boldsymbol{#1}}
\newcommand{\df}{\stackrel{\mathrm{def}}{=}}
\newcommand{\ma}{Monge--Amp\`{e}re }
\newcommand{\Bf}{Bellman function}
\newcommand{\Bfs}{Bellman functions}
\newcommand{\BMO}{{\rm BMO}}
\newcommand{\Oe}{\Omega_\varepsilon}
\newcommand{\const}{{\rm const}}
\def\sgn{\operatorname{sgn}}
\newcommand{\rn}{\mathbb{R}^n}

\newtheorem{theorem}{Theorem}[section]

\newtheorem{lemma}[theorem]{Lemma}
\newtheorem{corollary}[theorem]{Corollary}

\newtheorem*{theorem*}{Theorem}{\bf}{\it}
\newtheorem*{proposition*}{Proposition}{\bf}{\it}
\newtheorem*{observation*}{Observation}{\bf}{\it}
\newtheorem*{lemma*}{Lemma}{\bf}{\it}

\theoremstyle{definition}

\theoremstyle{remark}
\newtheorem{remark}[theorem]{Remark}

\numberwithin{equation}{section}

\allowdisplaybreaks
\begin{document}

\title{Sharp $L^p$ estimates on $\BMO$}

\author{Leonid Slavin}
\address{University of Cincinnati}
\email{leonid.slavin@uc.edu}

\author{Vasily Vasyunin}
\address{St. Petersburg Department of the V.~A.~Steklov
Mathematical Institute, RAS}
\email{vasyunin@pdmi.ras.ru}
\thanks{L. Slavin's research supported in part by the NSF (DMS-1041763)}

\thanks{V. Vasyunin's research supported in part by RFBR (08-01-00723-a)}

\subjclass[2000]{Primary 42A05, 42B35, 49K20}

\keywords{BMO, norm equivalence, explicit Bellman function, \ma equation}

\date{May 26, 2010}

\begin{abstract}
We construct the upper and lower Bellman functions for the $L^p$
(quasi)-norms of BMO functions. These appear as solutions to a series of 
\ma boundary value problems on a non-convex plane domain. The knowledge of the \Bfs\
leads to sharp constants in inequalities relating average oscillations of BMO functions and various BMO norms.
\end{abstract}
\maketitle
\section{Introduction}
For a measurable set $E\subset\mathbb{R}^n$ with finite, non-zero Lebesgue measure $|E|,$  and a locally integrable real-valued function $\varphi,$ let
$\av{\varphi}E$ denote the average of $\varphi$ over $E,$
$$
\av{\varphi}E=\frac1{|E|}\int_E\varphi.
$$
For a function $\varphi,$ a cube $Q,$ and $p>0$, let $\|\varphi\|_{L^p(Q)}$ denote the $Q$-normalized $L^p$ (quasi-)norm of $\varphi,$
$$
\|\varphi\|_{L^p(Q)}=\av{|\varphi|^p}Q^{1/p}.
$$
Observe that $\|\varphi\|_{L^p(Q)}$ is increasing in $p:$
\eq[d0.2]{
\|\varphi\|_{L^{p_1}(Q)}\le\|\varphi\|_{L^{p_2}(Q)}
}
for $p_1\le p_2,$ with equality happening if and only if $\varphi=\const$
on $Q.$

Fix a cube $Q$ in $\rn$ and let $\BMO^p(Q)$ be the (factor-)space
\eq[d1]{
\BMO^p(Q)=\{\varphi\in L^1(Q)\colon
\av{|\varphi-\av{\varphi}J|^p}J\le C^p<\infty,\
\forall~\text{cube}~J\subset Q\},
}
with the smallest such $C$ being the corresponding norm
(quasi-norm for $0<p<1$),
\eq[b21]{
\|\varphi\|_{\BMO^p(Q)}=
\sup_{\text{cube}~J\subset Q}\av{|\varphi-\av{\varphi}J|^p}J^{1/p}.
}

It is known that all $p$-based norms defined by \eqref{b21} are
equivalent and so \eqref{d1} defines the same space for all $p>0.$
This fact is usually seen as a consequence of the John--Nirenberg
inequality, although using that inequality to prove it will produce
suboptimal constants of norm equivalence. One of the primary motivations
of this work is to quantify this equivalence precisely, in dimension $1.$
To this end, we relate all $\BMO^p$ norms to the $\BMO^2$ norm. 
The reason $\BMO^2$ norm plays a central role here is that it allows
us to take advantage of the self-duality of $L^2(Q).$ From now on, we
reserve the name $\BMO$ for $\BMO^2:$
$$
\BMO(Q)=\{\varphi\in L^1(Q)\colon\av{\varphi^2}J-\av{\varphi}J^2\le
C^2<\infty,\ \forall~\text{cube}~J\subset Q\}.
$$
Thus, we would like to find the best
constants $c_p,\,C_p$ in the double inequalities
\eq[b22]{
c_p\|\varphi\|_{\BMO(Q)}\;\le\!\!\sup_{\text{interval}~J\subset Q}
\av{|\varphi-\av{\varphi}J|^p}J^{1/p}\;\le\; C_p\|\varphi\|_{\BMO(Q)}\,,
\qquad p>0.
}
Some cases are trivial: \eqref{d0.2} implies that $C_p=1$ for $p\le 2,$
and $c_p=1$ for $p\ge 2$ (the equalities hold, for instance, for any
$\varphi$ with zero average and constant modulus on $Q$). 

To find $c_p$ and $C_p,$ we estimate, for any $\varphi\in\BMO(Q),$ the quantity $\av{|\varphi-\av{\varphi}Q|^p}Q$ 
(we will refer to this as the $p$-oscillation of $\varphi$ over $Q$) in terms of its $2$-oscillation
$\av{\varphi^2}Q-\av{\varphi}Q^2$ and $\|\varphi\|_{\BMO(Q)}.$ Our estimates are sharp for any $p>0$ 
and one can use them to estimate the $p$-oscillation of a function in terms of its $q$-oscillation and BMO norm,
for any $0<p,q<\infty.$ 
What is surprising is that, due to the nature of the optimizers in the $p\leftrightarrow2$
inequalities, we obtain {\it sharp} $p\leftrightarrow q$ inequalities,
whenever $p\in[1,2]$ and $q\in[2,\infty).$

The oscillation estimates immediately yield the norm equivalence statement~\eqref{b22}.
However, the norm equivalence itself may or may not be sharp; at this point we can only show sharpness for $p>2.$

The principal step in getting oscillation estimates is to obtain sharp two-sided inequalities for the $p$-th power of
$L^p(Q)$-norms of $\BMO$ functions, i.e. expressions of the form
$\av{|\varphi|^p}Q.$ It turns out that these quantities
are always finite, meaning that $\BMO(Q)\subset L^p(Q),$ as sets.
These sharp $L^p$ estimates on BMO are our main goal.

Following the template of the John--Nirenberg project \cite{sv},
we define the {\it upper and lower \Bfs} for the problem: for $p>0$
and $\ve>0,$ let
\eq[b23]{
\bel{B}_{\ve,p}(x)=\sup_{\varphi\in \BMO_{\ve}(Q)}\left\{
\av{|\varphi|^p}Q\colon \av{\varphi}Q=x_1,\av{\varphi^2}Q=x_2\right\},
}
\eq[b24]{
\;\bel{b}_{\ve,p}(x)\;=\inf_{\varphi\in \BMO_{\ve}(Q)}\left\{
\av{|\varphi|^p}Q\colon \av{\varphi}Q=x_1,\av{\varphi^2}Q=x_2\right\},
}
where $\BMO_\ve(Q)$ is the $\ve$-ball in $\BMO(Q),$
$$
\BMO_\ve(Q)=\{\varphi\in\BMO(Q)\colon\|\varphi\|_{\BMO(Q)}\le\ve\}.
$$
It easy to check that these functions are independent of the interval
$Q.$ Their domain is
$$
\Oe=\{x:x_1^2\le x_2\le x_1^2+\ve^2\}.
$$
Indeed, for every $\varphi\in\BMO_\ve(Q)$ and every subinterval $J$
of $Q$ the corresponding {\it Bellman point}
$(\av{\varphi}J,\av{\varphi^2}J)$ is in $\Oe:$ the first inequality
is just H\"{o}lder's inequality and the second one holds since
$\|\varphi\|_{\BMO(Q)}\le\ve.$ This is the same domain as in~\cite{sv},
and, as was the case there, each \Bf\ will satisfy a \ma equation on
$\Oe.$ However, unlike the \Bfs\ in~\cite{sv}, the functions defined
by~\eqref{b23} and~\eqref{b24} do not have additive homogeneity, while
the domain $\Oe$ does not allow for multiplicative homogeneity. Thus we
cannot hope to reduce the partial differential equations to ordinary
ones, as was done in the previous work. Instead, we construct special
foliations of $\Oe$ by straight-line characteristics along which each
\Bf\ must be linear. These foliations, different for various ranges
of $p,$ allow us to solve the PDE, thus obtaining the \Bfs, and,
simultaneously, construct optimizers in the inequalities being proved.

Having introduced the main objects of study, let us say a bit about
the method. The term ``Bellman functions'' alludes to similar extremal
constructs in the dynamic programming of R.~Bellman~\cite{be}. There
are deep parallels between the Bellman frameworks of optimal stochastic control and harmonic
analysis (see, for example,~\cite{ntv3} and~\cite{vol}),
although we make no use of those connections here. In the mid- to
late 1980s, R.~Burkholder~(\cite{b1,b2}) started using specially
designed functions with delicate size and concavity properties to
prove sharp inequalities for martingales by induction on scales.
In the 1990s, the method took its modern form in the work of
F.~Nazarov, S.~Treil, and A.~Volberg, starting with~\cite{ntv1},
\cite{nt}, and \cite{ntv2}. However, the first exact \Bfs, as explicit
solutions of extremal problems such as~\eqref{b23} and~\eqref{b24},
did not appear until~\cite{v1}. Bellman analysis on $\BMO$ originated
with~\cite{sv} and continued in~\cite{v2}. Other notable explicit
\Bfs\ appeared in the work of A.~Melas and co-authors (see, for
instance,~\cite{melas1,melas2,mn}), although those functions were
found using combinatorial analysis of the operator in question,
the dyadic maximal function, as opposed to solving the Bellman
partial differential equation. The \Bfs\ for the maximal operator
were taken up again in~\cite{sst} and~\cite{ssv}. Those works,
together with~\cite{vv2}, initiated the study of the relationship
among explicit \Bfs, the \ma geometry of Bellman domains, and the
structure of optimizers in the corresponding inequalities. Most
recently, that line of investigation continued in~\cite{vv1}. 
The current work also presents that unified vision.

How do we proceed from estimates on $\av{|\varphi|^p}Q$ to the norm
equivalence statement \eqref{b22}? By the definitions of
$\bel{B}_\ve$ and $\bel{b}_\ve,$ we have, for every $\varphi\in\BMO_\ve(Q)$
and every subinterval $I$ of $Q,$
$$
\bel{b}_{\ve,p}(\av{\varphi}I,\av{\varphi^2}I)\le
\av{|\varphi|^p}I\le\bel{B}_{\ve,p}(\av{\varphi}I,\av{\varphi^2}I).
$$
Replacing $\varphi$ with $\varphi-\av{\varphi}I$ gives sharp oscillation estimates:
$$
\bel{b}_{\ve,p}(0,\av{\varphi^2}I-\av{\varphi}I^2)\le
\av{|\varphi-\av{\varphi}I|^p}I\le\bel{B}_{\ve,p}(0,\av{\varphi^2}I-
\av{\varphi}I^2).
$$
Analyzing each inequality separately, we get the desired norm inequalities:
$$
\bel{b}_{\|\varphi\|,p}(0,\|\varphi\|)\le
\sup_{\text{interval}~I\subset Q}\av{|\varphi-\av{\varphi}I|^p}I\le
\bel{B}_{\|\varphi\|,p}(0,\|\varphi\|),
$$
where $\|\cdot\|\df\|\cdot\|_{\BMO(Q)}.$

The functions $\bel{B}_{\ve,p}$ and $\bel{b}_{\ve,p}$ can be viewed
as special cases in a more general framework. Namely, take a function
$f$ on $\mathbb{R}$ and define, formally, the \Bfs\
\eq[b23.1]{
\bel{B}_{\ve,f}(x)=\sup_{\varphi\in \BMO_{\ve}(Q)}\left\{
\av{f(\varphi)}I\colon\av{\varphi}I=x_1,\av{\varphi^2}I=x_2\right\}
}
and
\eq[b24.1]{
\quad\bel{b}_{\ve,f}(x)\;=\inf_{\varphi\in \BMO_{\ve}(Q)}\left\{
\av{f(\varphi)}I\colon\av{\varphi}I=x_1,\av{\varphi^2}I=x_2\right\}.
}
(Here we explicitly allow for the possibility that one or both of
these functions take on infinite values.) We will see that such a
general view is beneficial: we will develop several canonical building
blocks, each defined in a sub-domain of $\Oe;$ these blocks, when
appropriately arranged and glued, produce the functions~\eqref{b23.1},
\eqref{b24.1} for various choices of $f,$ including the power function.
Equally important, the optimizers in the \Bfs\ --- those functions
$\varphi$ on which supremum or infimum is attained --- turn out to be
determined locally by the canonical blocks, rather than by what specific $f$
is being considered. In a very tangible sense, many $\BMO$ inequalities
have the same optimizers.

Let us outline several important choices of $f:$ $f(s)=|s|^p$ yields
definitions~\eqref{b23} and~\eqref{b24}; $f(s)=e^s$ produces the
Bellman setup for the integral John--Nirenberg inequality from~\cite{sv};
since $\lim_{p\to0}\av{|\varphi|^p}Q^{1/p}=e^{\av{\log |\varphi|}Q},$
the choice $f(s)=\log|s|$ gives the appropriate limiting setup
for~\eqref{b23} and~\eqref{b24}, although we will see that the resulting
``$\BMO^0,$'' is not, in fact, $\BMO.$

There are at least two more choices of importance:
$f(s)=\chi_{(-\infty,-\lambda]\cup[\lambda,\infty)}(s)$ for
$\lambda\ge0$ and $f(s)=e^{|s|}.$ The former yields the setup for
the classical, weak-form John--Nirenberg inequality and the latter,
for the two-sided integral John--Nirenberg inequality.  Each of these
two cases requires a slight modification of our building blocks.
The first was considered in~\cite{v2}, in a much more limited context,
and the second will be considered elsewhere.

What are the assumed conditions on $f?$ Since our focus in the present
work is on two-sided inequalities, we will assume that $f$ is even.
In order to simplify exposition we will also assume that $f$ is non-negative and smooth, except,
possibly, at $0.$

Although all definitions in this section are valid in any dimension,
at present we are only able to find the \Bfs\ in the one-dimensional
case, where all cubes in the definition of $\BMO$ are intervals.
We can obtain meaningful dimension-dependent estimates of the
norm-equivalence constants in higher dimensions, but it is apart from
our main interest here, which is in sharpness and explicit \Bfs.
In fact, we hold out the possibility that the \Bfs\ --- and so the norm
estimates --- are {\it dimension-free}, which would have major
implications for analysis on $\BMO$ and related function classes
(such as $A_p$).  At this time, we cannot show it, since a key geometric
ingredient in our proofs works only in dimension $1;$ we hope to be able
to give a definitive answer for higher dimensions in the future.

The geometry of \ma foliations plays a central role in the construction
of \Bfs\ and their optimizers. As such, it is given a central role in
our exposition. The picture of a foliation concisely captures the
nature of the extremal problem at hand, and it is those foliations,
first built locally and then carefully glued together, that we would
like the reader to remember. While we provide precise algebraic
descriptions of the \Bfs, our proofs often appeal to their geometric
nature. For example, when building optimizers, it is certainly possible
to show that a given function is in $\BMO_\ve$ by a direct calculation.
However, it is more geometrically meaningful --- and often simpler --- to show that
all of its Bellman points are in $\Oe.$

Lastly, proving sharp norm estimates, while a lofty goal, does not
require one to know the origin of the \Bfs\ or of their optimizers;
a relatively straightforward verification that these are, in fact,
optimal would suffice. However, we choose to present their construction
in considerable detail. This serves our second major goal: to provide
a complete account of the modern Bellman--\ma approach to problems with
non-convex Bellman domains. Many other traditional harmonic analysis
questions give rise to such domains and much of what follows should
be applicable there.

The paper is organized as follows: in Section~\ref{inequalities}, we
state the sharp inequalities proving which was our main motivation;
in Section~\ref{equations}, we derive the boundary value problems
our \Bfs\ should solve and outline the geometric approach to
finding the solutions; in Section~\ref{local} we construct local Bellman {\it candidates} for several subdomains of $\Oe;$ these are glued together to produce global solutions in Section~\ref{global};
in Sections~\ref{induction} and~\ref{converse} we use induction on scales and optimizers to prove that the global candidates are, in fact, the true \Bfs; in
Section~\ref{proofs} we provide the proofs of the inequalities from
Section~\ref{inequalities}; finally, in Section~\ref{other}, we briefly consider several additional
choices of $f$ in~\eqref{b23.1} and \eqref{b24.1} and state the
corresponding \Bfs.
\section{Sharp inequalities}
\label{inequalities}
Although the main results of this work are the explicit expressions for
the functions~\eqref{b23} and~\eqref{b24} for various ranges of $p,$
these are relatively complicated and stated in Section~\ref{global},
after the relevant notation has been introduced and various blocks
that comprise these expressions have been developed. In this section,
we state the immediate consequences of our knowing the Bellman functions:
the appropriate sharp inequalities for $\BMO.$ Since we know both 
the upper and lower functions, we naturally get two-sided inequalities; even though some are elementary,
they are included for symmetry and to emphasize their source.
\begin{theorem}
\label{th1.0}
For an interval $Q$ and any $\varphi\in\BMO(Q)$ such that $\|\varphi\|_{\BMO(Q)}\ne0$ we have:
$$
\begin{array}{ll}
\text{if}~0< p\le1,
& 2^{p-2}\|\varphi\|^{p-2}_{\BMO(Q)}(\av{\varphi^2}Q-\av{\varphi}Q^2)\le
\av{|\varphi-\av{\varphi}Q|^p}Q\le\Big(\av{\varphi^2}Q-\av{\varphi}Q^2\Big)^{p/2};
\\
&\\
\text{if}~1\le p\le 2,
&\frac p2\,\Gamma(p)\|\varphi\|^{p-2}_{\BMO(Q)}(\av{\varphi^2}Q-\av{\varphi}Q^2)\le
\av{|\varphi-\av{\varphi}Q|^p}Q\le\Big(\av{\varphi^2}Q-\av{\varphi}Q^2\Big)^{p/2};
\\
&\\
\text{if}~2\le p<\infty ,
&\Big(\av{\varphi^2}Q-\av{\varphi}Q^2\Big)^{p/2}\le
\av{|\varphi-\av{\varphi}Q|^p}Q
\le
\frac p2\,\Gamma(p)\|\varphi\|^{p-2}_{\BMO(Q)}(\av{\varphi^2}Q-\av{\varphi}Q^2),
\end{array}
$$
and these inequalities are sharp and attainable.
\end{theorem}
As a straightforward corollary of this theorem we obtain inequalities relating $\BMO^2$ and $\BMO^p$ norms; for $p>2$ these are the best possible.
\begin{theorem}
\label{th1.1}
For an interval $Q$ and $\varphi\in\BMO(Q),$ we have:
$$
\begin{array}{ll}
\text{if}~0< p\le1,
& 2^{1-2/p}\|\varphi\|_{\BMO(Q)}\le
\|\varphi\|_{\BMO^p(Q)}\le\|\varphi\|_{\BMO(Q)};
\\
&\\
\text{if}~1\le p\le 2,
&\left(\frac p2\,\Gamma(p)\right)^{1/p}\|\varphi\|_{\BMO(Q)}\le
\|\varphi\|_{\BMO^p(Q)}\le\|\varphi\|_{\BMO(Q)};
\\
&\\
\text{if}~2\le p<\infty ,
&\|\varphi\|_{BMO(Q)}\le\|\varphi\|_{\BMO^p(Q)}\le
\left(\frac p2\,\Gamma(p)\right)^{1/p}\|\varphi\|_{\BMO(Q)}.
\end{array}
$$
The right-hand side inequalities for $p<2$ and both left- and right-hand side inequalities for $p>2$ are sharp and attainable.
\end{theorem}
\begin{remark}
At this point we do not know if the left-hand inequalities for $p<2$ are sharp.
\end{remark}
Of course, Theorem~\ref{th1.1} allows us to relate different $\BMO^p$
norms bypassing $\BMO^2,$ but the resulting inequalities may no longer be sharp. 
However, if we instead use Theorem~\ref{th1.0} and our knowledge of the optimizers in its statements, we get the best possible inequalities relating $p$-oscillations for certain ranges of the parameter $p.$ Specifically, we have
\begin{theorem}
\label{th2.1}
Fix an interval $Q$ and numbers $p_1\in[1,2]$ and $p_2\in[2,\infty).$ Then, for any $\varphi\in\BMO(Q),$ we have
$$
\av{|\varphi-\av{\varphi}Q|^{p_1}}Q^{p_2/p_1}\le\av{|\varphi-\av{\varphi}Q|^{p_2}}Q\le\frac{p_2\Gamma(p_2)}{p_1\Gamma(p_1)}\|\varphi\|^{p_2-p_1}_{\BMO(Q)}\av{|\varphi-\av{\varphi}Q|^{p_1}}Q,
$$
and these inequalities are sharp and attainable.
\end{theorem}
As mentioned in the introduction, in this paper we do not attempt to
prove sharp two-sided John--Nirenberg inequality, that is we do not find
the Bellman functions~\eqref{b23.1} and \eqref{b24.1} for $f(s)=e^{|s|}.$
However, because $e^{|s|}=1+|s|+\sum_{k=2}^\infty\frac{|s|^k}{k!}$ and
{\it all} upper Bellman functions $\bel{B}_{p,\ve}$ for $p\ge2$ turn
out to have the same optimizers, we can sharply estimate
$$
\av{e^{|\varphi-\av{\varphi}{\scriptscriptstyle Q}|}}Q-
\av{|\varphi-\av{\varphi}Q|}Q.
$$
Coupling the resulting estimate with the inequality $\av{|\varphi-\av{\varphi}Q|}Q\le\|\varphi\|_{\BMO(Q)},$  we obtain the following ``almost sharp,''
two-sided integral John--Nirenberg inequality:
\begin{theorem}
\label{th3}
Take an interval $Q$ and let $\varphi\in\BMO(Q)$ be such that
$\|\varphi\|=\|\varphi\|_{\BMO(Q)}<1.$ Then we have
$$
\av{e^{|\varphi-\av{\varphi}{\scriptscriptstyle Q}|}}Q\le C(\|\varphi\|).
$$
Here the bound $1$ on $\|\varphi\|_{\BMO}$ is sharp and the best
\textup(smallest\textup) value of $C(\|\varphi\|)$ satisfies
\eq[ss6]{
\frac{1-\frac{\|\varphi\|}2}{1-\|\varphi\|}\le C(\|\varphi\|)\le
\frac{1-\frac{\|\varphi\|^2}2}{1-\|\varphi\|}.
}
\end{theorem}

By ``sharp'' we mean that there exist functions $\varphi$ with norm
$1$ for which the inequality fails. The reader can compare this result
with the one in~\cite{sv}, where the sharp {\it one-sided}
John--Nirenberg inequality was proved. One can get a sub-optimal
two-sided inequality from a one-sided one by simply doubling the
constant. If we do that with the result in~\cite{sv}, we get
$$
C(\|\varphi\|)\le\frac{2e^{-\|\varphi\|}}{1-\|\varphi\|},
$$
which is worse than in \eqref{ss6}.

It will be easy to prove all the inequalities stated in this section
after the explicit \Bfs\ \eqref{b23} and \eqref{b24}, as well as their
optimizers, have been found. We provide the proofs in
Section~\ref{proofs}.
\section{Boundary value problems and \ma foliations}
\label{equations}
\subsection{The equations and boundary conditions}
Take an interval $I$ and split it into two non-intersecting subintervals:
$I=I_-\cup I_+.$ For any sufficiently integrable function $\varphi$ on
$I,$
$$
\av{f(\varphi)}I=\alpha_-\av{f(\varphi)}{I_-}+
\alpha_+\av{f(\varphi)}{I_+},
$$
where $\alpha_\pm=|I_\pm|/|I|.$
Fix two points $x^-\!,\,x^+\in\Oe$ such that
$\alpha_-x^-+\alpha_+x^+\in\Oe$ and consider a sequence of functions
$\{\varphi_n\}$ such that $\varphi_n\in\BMO_\ve(I_-)\cap\BMO_\ve(I_+),$
$(\av{\varphi_n}{I_\pm},\av{\varphi_n^2}{I_\pm})=x^\pm,$ and
$\lim_{n\to\infty}\av{f(\varphi_n)}{I^\pm}=\bel{B}_{\ve,f}(x^\pm).$
This gives
$$
\lim_{n\to\infty}\av{f(\varphi_n)}I=
\alpha_-\bel{B}_{\ve,f}(x^-)+\alpha_+\bel{B}_{\ve,f}(x^+).
$$
If each $\varphi_n$ could be chosen so that $\varphi_n\in\BMO_\ve(I),$
then we could conclude that
\eq[f1]{
\bel{B}_{\ve,f}(\alpha_-x^-+\alpha_+x^+)\ge
\alpha_-\bel{B}_{\ve,f}(x^-)+\alpha_+\bel{B}_{\ve,f}(x^+).
}
Although $\alpha_-x^-+\alpha_+x^+\in\Oe,$ in general
$$
\BMO_\ve(I)\subsetneq\BMO_\ve(I_-)\cap\BMO_\ve(I_+)\cap
\{\varphi\colon\av{\varphi^2}I -\av{\varphi}I^2\le\ve^2\},
$$
which is a major difference between continuous and dyadic $\BMO.$
We will, nonetheless, enforce condition~\eqref{f1} for the upper \Bf\
candidate $B$ and the converse inequality  for the lower candidate $b.$
Thus, we look for functions $B$ and $b$ with the property that for all pairs
of points $x^\pm\in\Oe$ such that the whole line segment $[x^-\!,\,x^+]$
is in $\Oe,$ we have
$$
B(\alpha_-x^-+\alpha_+x^+)\ge\alpha_-B(x^-)+\alpha_+B(x^+)
$$
and
$$
b(\alpha_-x^-+\alpha_+x^+)\le\alpha_-b(x^-)+\alpha_+b(x^+),
$$
for all $\alpha_\pm>0$ with $\alpha_++\alpha_-=1.$ In other words,
we look for $B$ and $b$ that are concave and, respectively, convex
on any convex portion of $\Oe.$ We will refer to such functions as {\it locally concave} and
{\it locally convex}, respectively. If we also assume sufficient
differentiability on $B$ and $b,$ we get differential analogs of
these finite-difference inequalities:
$$
-\frac{d^2B}{d x^2}\ge0,\qquad \frac{d^2b}{d x^2}\ge0\qquad\text{in}~\Oe.
$$
In yet another restriction, the way \Bf\ candidates are used in
subsequent proofs suggests that we need to require that the candidates'
concavity/convexity be degenerate, i.e. we require that
$$
\det\left(\frac{d^2B}{d x^2}\right)=0,\qquad
\det\left(\frac{d^2b}{d x^2}\right)=0\qquad\text{in}~\Oe.
$$
We have natural boundary conditions for the candidates. Observe that
if $\av{\varphi^2}I=\av{\varphi}I^2,$ then $\varphi$ is constant on
$I$ and so $\av{f(\varphi)}I=f(x_1),$ giving $\bel{B}(x_1,x_1^2)=f(x_1)$ and
$\bel{b}(x_1,x_1^2)=f(x_1).$ In addition, since $f$ is assumed even, the
\Bfs~\eqref{b23.1} and \eqref{b24.1} do not change if we replace $\varphi$
with $-\varphi$ in their definitions and so
$$
\bel{B}(x_1,x_2)=\bel{B}(|x_1|,x_2),\quad
\bel{b}(x_1,x_2)=\bel{b}(|x_1|,x_2).
$$
Accordingly, it is enough to construct candidates $B,$ $b$ in the
half-domain $\Oe^+\df\Oe\cap\{x_1\ge0\}.$ Because of the symmetry,
we impose a zero Neumann condition on the ``internal'' boundary $x_1=0.$

In what follows we will use the notation
$g_{z}\df\frac{\partial g}{\partial z}$ for any function $g$ and
variable $z.$ Thus, we set out to solve the following boundary
value problems for candidates $B$ and $b:$
\eq[f5]{
\begin{gathered}
B_{x_1x_1}B_{x_2x_2}=B^2_{x_1x_2},\quad B_{x_1x_1}\le0,~B_{x_2x_2}\le0
\quad \text{in}~\Oe^+;
\\
B(x_1,x_1^2)=f(x_1),\quad B_{x_1}|_{x_1=0}=0,
\end{gathered}
}
\eq[f6]{
\begin{gathered}
b_{x_1x_1}b_{x_2x_2}=b^2_{x_1x_2},\quad b_{x_1x_1}\ge0,~b_{x_2x_2}\ge0
\quad \text{in}~\Oe^+;
\\
b(x_1,x_1^2)=f(x_1),\quad b_{x_1}|_{x_1=0}=0.
\end{gathered}
}
\subsection{\ma equations and their solutions}
In this part, we state a general result that will help us solve
equations~\eqref{f5} and \eqref{f6}. A homogeneous \ma equation says
that at every point in the domain the Gaussian curvature of the
surface defined by the solution is zero, i.e. there is a direction
along which the solution is linear to the second order. Such directions
form a vector field and the integral curves of this vector field turn
out to be straight lines. These linear trajectories foliate the
whole domain and --- unless the defect of the Hessian is allowed
to exceed 1 at a particular point --- do not intersect.

There exists a formal way of obtaining \ma solutions
as functions that are linear along certain trajectories. From this
viewpoint, describing a foliation of the domain by such trajectories
({\it a Bellman foliation}, in our parlance) is equivalent to
determining the corresponding solution uniquely. The theorem below
formalizes this method of foliations. We only consider the case of
plane domains here. The general formulation, as well as a simple proof,
can be found in~\cite{vv1}.
\begin{theorem}
Let $\Omega$ be a plane domain and $G=G(x_1,x_2)$ be a $C^2$ function
satisfying the homogeneous \ma equation in $\Omega:$
$$
G_{x_1x_1}G_{x_2x_2}=G^2_{x_1x_2},
$$
and such that either $G_{x_1x_1}\ne0$ or $G_{x_2x_2}\ne0.$

Let
\eq[ma1]{
t_1=G_{x_1};\quad t_2=G_{x_2};\quad t_0=G-t_1x_1-t_2x_2.
}

Then the functions $t_k$ are constant on each integral trajectory
generated by the kernel of the Hessian $\frac{d^2G}{dx^2}.$ Moreover,
these integral trajectories are straight lines given by
\eq[ma3]{
x_1dt_1+x_2dt_2+dt_0=0.
}
\end{theorem}
In what follows we will refer to the trajectories~\eqref{ma3} as
{\it extremal trajectories}, since they are used not only to find
a suitable \Bf\ candidate, but also to build extremizing
functions/sequences proving that the candidate is, indeed, optimal.
\subsection{Main empirical principles}
According to the preceding discussion, we need to look for a candidate
$G=G(x_1,x_2)$ satisfying either~\eqref{f5} or~\eqref{f6}. It would
be desirable to have the Bellman function for each $f$ given by a
single $C^2$ expression that is either concave or convex in the whole
$\Oe.$ However, as we will see shortly, this is asking for too much.
Instead, we will have our Bellman functions built out of several
canonical \ma solutions, each $C^2$ in a portion of $\Oe,$ glued
together so as to preserve the sign of the generalized second
differential. A word about the nomenclature: $G$ will stand for a
generic Bellman candidate, whether upper or lower, whether defined
on the whole domain or its part; as before, $B$ will stand for a
similarly generic upper candidate, and $b$ for a lower one; when
the emphasis is on the geometry of a canonical solution, the solution
will be designated with its own name and necessary indices, as will
be the sub-domain on which the solution is built; finally, we will omit the indices $f$ and
$\ve$ when no ambiguity arises.

We will build our canonical solutions by constructing the foliations
of the corresponding portions of the domain by extremal trajectories,
while adhering to several empirical principles. Since the \ma
apparatus~\eqref{ma1}--\eqref{ma3} does not differentiate among
various solutions of~\eqref{f5} or~\eqref{f6}, additional arguments,
having to do with the nature of the \Bf, are needed, chief among
which is how extremal trajectories are used to construct optimizers.
These are our principles:

{\it Principle 1: Symmetry.}
Since each \Bf\ is even in $x_1,$ each Bellman foliation is symmetric
with respect to the line $x_1=0.$

Accordingly, we only need to build each Bellman foliation in the
half-domain $\Oe^+,$ but have to consider carefully what happens
on the internal boundary $x_1=0.$ In which way can a trajectory
intersect this boundary? Since the picture is symmetric, any such
trajectory will have its counterpart from the domain
$\Oe\cap\{x_1\le0\}$ hitting the same point $(0,x_2).$ When the
domain $\Oe$ is considered as a whole, we will have two trajectories
intersecting at a point. This means that either they are two halves
of the same horizontal trajectory or the Hessian has defect 2 at
the point, meaning it is the identically zero matrix, and so the
\Bf\ is a linear function near that point. Indeed, we will encounter
both situations below. In either case, since $G_{x_1}=t_1=0$ when
$x_1=0,$ the function $G$ depends only on $x_2$ in the subdomain
that includes the line $x_1=0.$

{\it Principle 2: Tangency.} Any trajectory intersecting the upper
boundary of $\Oe,$ $x_2=x_1^2+\ve^2\!,$ must do so tangentially (unless
the \Bf\ is linear in the part of the domain containing the point of
tangency, as in this case any straight line is an extremal trajectory).

Let us explain this principle: observe that once the Bellman foliation
for the problem is determined, every point of the domain can be put on
one of the trajectories. That point, say $x,$ prescribes the averages
over the interval $Q$ of each of the functions over which the extremum
is taken in definitions~\eqref{b23} and \eqref{b24}. If such a function
is optimal (or close to optimal), then when the interval is split into
subintervals, $Q=Q_-\cup Q_+,$ the pairs of averages $x^\pm$ should
remain on the same trajectory. In such a split, $x$ will be located
between $x^-$ and $x^+.$ We must, therefore, have a tangential
intersection of the trajectory with the upper boundary, otherwise,
if we take $x^\pm$ close enough to $x,$ one of the endpoints will
exit $\Oe.$

{\it Principle 3. Optimality.} The upper \Bf\ is the smallest locally
concave solution of the \ma equation, while the lower \Bf\ is the
largest locally convex solution.

This principle may be intuitively clear, and it will be rigorously
demonstrated in Section~\ref{induction}, where Bellman induction
on scales is used to show that any locally concave function $B$
satisfying $B(x_1,x_1^2)=f(x_1)$ is a pointwise majorant of
$\bel{B}_{f,\ve},$ while any locally convex solution $b$ satisfying
$b(x_1,x_1^2)=f(x_1)$ is a pointwise minorant of $\bel{b}_{f,\ve}.$
The terms ``super-solution'' and ``sub-solution,'' respectively, are
typically used for such candidates.
\section{Local Bellman candidates: the four building blocks}
\label{local}
Having laid down our basic principles, we start building the foliations
(and so the candidates) that comply with these principles. Since we know
the behavior of any extremal trajectory touching the upper boundary, it
is convenient to start with such trajectories. The tangent line at a
point $(a,a^2+\ve^2)$ is given by
\eq[ff1]{
x_2=2ax_1+\ve^2-a^2.
}
Each such tangent intersects the lower boundary $x_2=x_1^2$ at two points,
$$
u_\pm=a\pm\ve.
$$
Let us show that if the foliation being built includes {\it a family} of
such tangents,  the whole tangent line cannot be a single extremal
trajectory, since $G_{x_2x_2}$ changes sign depending on whether $x$
is to the left or to the right of $(a,a^2+\ve^2).$ Indeed, for every
$x$ on the tangent line~\eqref{ff1}, $t_1,$ $t_2,$ and $t_0$ are functions
of $a$ only. Recall that $t_2=G_{x_2}.$ Then $G_{x_2x_2}=t_2'(a)a_{x_2}.$
Since $a$ is fixed along this line, $t_2'(a)$ is constant.
From~\eqref{ff1},
\eq[ff0]{
a_{x_2}=\frac1{2(x_1-a)}\,,
}
which changes sign at $x_1=a.$

Therefore, if the foliation contains a family of such tangents (which
justifies differentiating with respect to $a$ above), each extremal
trajectory  continues either to the right of the point of tangency
or to the left, but not both. This is in contrast to the situation,
also considered below, when the sub-domain being foliated lies entirely
under {\it one} two-sided tangent.

We will now consider the four sub-foliations out of which we will later build
two complete Bellman foliations (one for the upper and one for the
lower function) for various ranges of $p.$ The sub-foliations are two
families of one-sided tangents and two ``phase transition regimes'' used
to connect those families smoothly. The four are, in order of presentation:
\ben
\item
A collection of one-sided tangents for which the point of their
intersection with the lower boundary is to the right of the point
of tangency;
\item
A collection of one-sided tangents whose point of intersection
with the lower boundary is to the left of the point of tangency;
\item
Any foliation of the convex compact set lying under a single
two-sided tangent. There are several ways to foliate such a set,
depending on its location; in each case the bounding tangent is an
element of the foliation;
\item
Any foliation of a curvilinear ``triangle'' located between two
differently oriented one-sided tangents sharing a point on the
lower boundary. As we will see, in such a triangle the Bellman
function must be linear and so the very notion of a Bellman
foliation is trivial, as every straight line is an extremal trajectory.
\een
\begin{remark}
It would, of course, be preferable to not have to combine various
solutions, instead having a single family of tangents as the Bellman
foliation on each side of the line $x_1=0$ with a single transition
regime containing that line and connecting the two families. Such a
situation does, in fact, occur for $f(s)=|s|^p,$ $p\ge 1.$ However,
for $0<p<1$ the solutions corresponding to each family of tangents
change concavity at various points throughout the domain, necessitating
the introduction of other transition regimes.
\end{remark}
\subsection{The tangents with $u=a+\ve$}
Here we consider a family of tangents~\eqref{ff1} with $a\in[a_1,a_2],$
or $u\in[u_1,u_2],$ $u_i=a_i+\ve.$ The tangents under consideration are
one-sided, extending to the right of the point of tangency, i.e. we
have $x_1\in[a,u].$ The Bellman candidate built along these trajectories
will be called $F^+$ (or, more fully, $F^+(x;u_1,u_2)$) and the part of
$\Oe$ so foliated, $\Omega_{F^+}(u_1,u_2).$ The foliation is shown in
Figure~\ref{fig 0.001}.
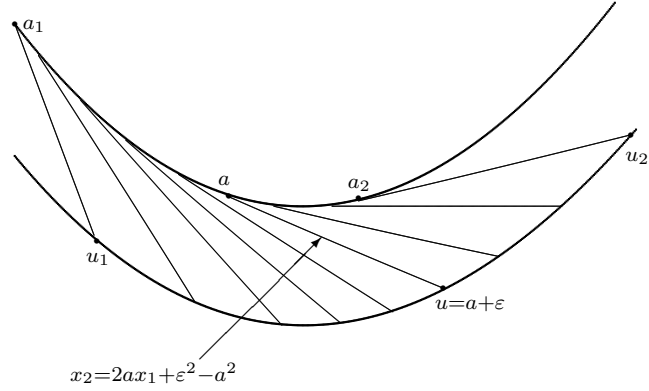
\begin{figure}[ht]
\begin{picture}(200,150)
\thicklines
\qbezier(-10,90)(100,-43)(225,100)
\qbezier(-10,140)(100,-2)(217,148)
\put(10,5){$\scriptstyle x_2=2ax_1+\ve^2-a^2$}
\put(66,79){$\scriptstyle a$}
\put(-8,138){$\scriptstyle a_1$}
\put(115,78){$\scriptstyle a_2$}
\put(16,50){$\scriptstyle u_1$}
\put(220,88){$\scriptstyle u_2$}
\put(148,32){$\scriptstyle u=a+\ve$}
\thinlines
\put(60,13){\vector(1,1){46}}
\put(71,75){\circle*{2}}
\put(152,40){\circle*{2}}
\put(-10,140){\circle*{2}}
\put(120,74){\circle*{2}}
\put(21,58){\circle*{2}}
\put(223,98){\circle*{2}}
\qbezier(70,75)(111,57)(152,40)
\qbezier(52,83)(92,57)(133,31)
\qbezier(32,96)(72,61)(113,27)
\qbezier(15,111)(53,68)(91,26)
\qbezier(-1,128)(30,80)(58,35)
\qbezier(-10,140)(5,100)(20,60)
\qbezier(88,71)(130,61)(173,52)
\qbezier(110,71)(153,71)(197,71)
\qbezier(120,73)(171,85)(223,98)
\end{picture}
\caption{The foliation of $\Omega_{F^+}(u_1,u_2).$}
\label{fig 0.001}
\end{figure}

The candidate $F^+$ is linear along each trajectory~\eqref{ff1} and
satisfies $F^+(u,u^2)=f(u),$ thus,
\eq[ff007]{
F^+(x_1,2ax_1+\ve^2-a^2)=m(u)(x_1-u)+f(u).
}
Let us calculate $t_2=F^+_{x_2}.$ Using~\eqref{ff0} and the equality
$u_{x_2}=a_{x_2},$ we get
$$
t_2=\frac{m'(u)(x_1-u)-m(u)+f'(u)}{2(x_1-a)}=
\frac12m'(u)+\frac{f'(u)-\ve m'(u)-m(u)}{2(x_1-u+\ve)}\,.
$$
Since $t_2$ is fixed whenever $a$ (or $u$) is fixed, this gives
\eq[ff008]{
t_2(u)=\frac12m'(u)
}
and
\eq[ff009]{
\ve m'(u)+m(u)=f'(u).
}
Solving equation \eqref{ff009} yields
$$
m(u)=e^{-u/\ve}
\Bigl(C+\frac1\ve\,\int_{u_1}^u\!\! f'(s)e^{s/\ve}\,ds\Bigr),
$$
and so we obtain our solution, defined for $x\in\Omega_{F^+}(u_1,u_2),$
\eq[ff4.0]{
F^+(x;u_1,u_2)=e^{-u/\ve}
\Bigl(C+\frac1\ve\int_{u_1}^u\!\!f'(s)e^{s/\ve}\,ds\Bigr)(x_1-u)+f(u),
}
where $u$ is given as a function of $x$ by
\eq[ff5]{
u=u_+=x_1+\ve-\sqrt{\ve^2-x_2+x_1^2}.
}
The last equation comes from plugging $u=a+\ve$ into~\eqref{ff1}, and the minus sign
in the front of the square root was chosen, because $x_1\in[a,u].$

We will be concerned with whether~\eqref{ff4.0} gives a concave or
convex candidate for each specific choice of $f$ (and in what parts
of the domain). To that end, let us check the sign of $F^+_{x_2x_2}.$
Using~\eqref{ff008} and~\eqref{ff009}, we obtain
$$
F^+_{x_2x_2}=t_2'(u)u_{x_2}=\frac12m''(u)u_{x_2}=
\frac1{2\ve}\left(f''(u)-f'(u)/\ve+m(u)/\ve\right)u_{x_2}.
$$
Since $u_{x_2}>0$ by~\eqref{ff0},  we conclude that
$\sgn(F^+_{x_2x_2})=\sgn(\tau_+),$ where
\eq[ff35.9]{
\tau_+(u)\df\ve e^{u/\ve}m''(u)=
(f''(u)-f'(u)/\ve)e^{u/\ve}+\frac1\ve C+
\frac1{\ve^2}\int_{u_1}^uf'(s)e^{s/\ve}\,ds.
}
Integrating by parts twice gives
\eq[ff36]{
\tau_+(u)=(f''(u_1)-f'(u_1)/\ve)e^{u_1/\ve}+\frac1\ve C+
\int_{u_1}^uf'''(s)e^{s/\ve}\,ds.
}
\subsection{The tangents with $u=a-\ve$}
We now switch to considering those one-sided tangents whose point of
intersection with the bottom boundary curve lies to the left of the
point of tangency. The picture is symmetric with the previous
one with respect to the $x_2$-axis. Thus, we replace in the preceding
formulas $x_1$ by $-x_1,$ $u$ by $-u,$ $a$ by $-a,$ $f(s)$ by $f(-s),$
$u_1$ by $-u_2,$ and $u_2$ by $-u_1.$ It is possible, of course, to
derive all formulas independently, but we only indicate which changes
are necessary. The solution constructed will be called $F^-$
(alternatively, $F^-(x;u_1,u_2)$) and the portion of $\Oe$ being
foliated, $\Omega_{F^-}(u_1,u_2).$ The foliation is shown on
Figure~\ref{fig 0.002}.
\begin{figure}[ht]
\begin{picture}(200,150)
\thicklines
\qbezier(210,90)(100,-43)(-25,100)
\qbezier(210,140)(100,-2)(-17,148)
\put(110,5){$\scriptstyle x_2=2ax_1+\ve^2-a^2$}
\put(128,80){$\scriptstyle a$}
\put(80,76){$\scriptstyle a_1$}
\put(198,139){$\scriptstyle a_2$}
\put(-30,90){$\scriptstyle u_1$}
\put(178,50){$\scriptstyle u_2$}
\put(25,31){$\scriptstyle u=a-\ve$}
\thinlines
\put(130,13){\vector(-1,1){42}}
\put(47,40){\circle*{2}}
\put(131,76){\circle*{2}}
\put(80,74){\circle*{2}}
\put(210,140){\circle*{2}}
\put(-23,98){\circle*{2}}
\put(180,59){\circle*{2}}
\qbezier(130,75)(89,57)(48,40)
\qbezier(148,83)(108,57)(67,31)
\qbezier(168,96)(128,61)(87,27)
\qbezier(185,111)(147,68)(109,26)
\qbezier(201,128)(170,80)(142,35)
\qbezier(210,140)(195,100)(180,60)
\qbezier(112,71)(70,61)(27,52)
\qbezier(90,71)(47,71)(3,71)
\qbezier(80,73)(29,85)(-23,98)
\end{picture}
\caption{The foliation of $\Omega_{F^-}(u_1,u_2).$}
\label{fig 0.002}
\end{figure}

Equations~\eqref{ff007} and~\eqref{ff008} remain the same, but
equation~\eqref{ff009} changes to
\eq[ff010]{
\ve m'(u)-m(u)=-f'(u),
}
which gives
$$
m(u)=e^{u/\ve}\Bigl(C+\frac1\ve\int_u^{u_2}\!\!f'(s)e^{-s/\ve}\,ds\Bigr).
$$

This yields our canonical solution, defined for
$x\in\Omega_{F^-}(u_1,u_2):$
$$
F^-(x;u_1,u_2)=e^{u/\ve}
\Bigl(C+\frac 1\ve\int_u^{u_2}\!\! f'(s)e^{-s/\ve}\,ds\Bigr)(x_1-u)+f(u),
$$
with
$$
u=u_-=x_1-\ve+\sqrt{\ve^2-x_2+x_1^2}.
$$
As before, the sign of $F^-_{x_2x_2}$ will be of interest. We have
$$
F^-_{x_2x_2}=t_2'(u)u_{x_2}=\frac12m''(u)u_{x_2}=
\frac1{2\ve}\bigl(-f''(u)-f'(u)/\ve+m(u)/\ve\bigr)u_{x_2}.
$$
Since $u_{x_2}<0$ by~\eqref{ff0}, we have
$\sgn(F^-_{x_2x_2})=\sgn(\tau_-),$ where
\eq[ff16]{
\tau_-(u)\df-\ve e^{-u/\ve}m''(u)=
\big(f''(u)+f'(u)/\ve\big)e^{-u/\ve}-
\frac1\ve C-\frac1{\ve^2}\int_u^{u_2}\!\!f'(s)e^{-s/\ve}\,ds.
}
Integrating by parts twice gives
\eq[ff37]{
\tau_-(u)=(f''(u_2)+f'(u_2)/\ve)e^{-u_2/\ve}-
\frac1\ve C-\int_u^{u_2}\!\!f'''(s)e^{-s/\ve}\,ds.
}
\subsection{The region under a two-sided tangent}
Next, fix $a$ and consider the tangent $x_2=2ax_1+\ve^2-a^2.$ It
intersects the lower boundary at the points $(a\pm\ve,(a\pm\ve)^2)$ and bounds
the convex subset of $\Oe$
$$
\Omega_L(a)=\{x_1^2\le x_2\le 2ax_1+\ve^2-a^2\}.
$$
(In our nomenclature, each domain is designated by the name of the
corresponding Bellman candidate; in a bit of inelegance, here we
have to define and name the domain first. Naturally, we reserve the
name $L$ for each candidate built on $\Omega_L.$)

In order for any foliation of $\Omega_L(a)$ to be a part of a symmetric
foliation of the whole $\Oe,$ we must have either $a=0$ (and so the
condition $L_{x_1}|_{x_1=0}=0$ must come into play) or $|a|-\ve\ge0,$
i.e. $\Omega_L(|a|)\subset\Oe^+.$

Let us first construct the solution for the set
$\Omega_L(0)=\{x_1^2\le x_2\le \ve^2\}.$ The Bellman candidate here
will be called $L_0.$ As explained in the formulation of the symmetry
principle, $L_0$ must be a function of $x_2$ only. Thus we have
$L_0(x)=g(x_2)$ for some function $g.$ On the boundary,
$L_0(x_1,x_1^2)=g(x_1^2)=f(x_1)$ and so $g(x_2)=f(\sqrt{x_2}).$ Thus,
\eq[g1]{
L_0(x)=f(\sqrt{x_2}).
}
The corresponding foliation of $\Omega_L(0)$ consists of horizontal lines.

We now fix an $a\ge\ve$ and construct a solution (there are actually
two) in the region $\Omega_L(a).$ Assume that there is a family of
trajectories foliating (a part of) this domain. Say such a
trajectory intersects the lower boundary at the point $(u,u^2)$ and
we have a whole interval of such values $u.$ We do not want our
trajectories to intersect the upper boundary of $\Omega_L(a),$ the
line $x_2=2ax_1+\ve^2-a^2,$ since we want to have that line as one
of the trajectories. Thus each trajectory exits $\Omega_L(a)$ through
a point $(v,v^2)$ on the lower boundary with $v=v(u).$ Then the
trajectory is given by
\eq[e1]{
x_2=(v+u)x_1-vu,
}
and we have, for a candidate $L,$
$$
L(x_1,(v+u)x_1-vu)=m(u)(x_1-u)+f(u).
$$
On the other hand, we must have $L(v,v^2)=f(v),$ hence,
\eq[e3]{
m=\frac{f(v)-f(u)}{v-u}.
}

As before, let us find $t_2=L_{x_2}.$ We have
$t_2=\bigl(m'(x_1-u)-m+f'(u)\bigr)u_{x_2}.$ Using~\eqref{e1} gives
$$
u_{x_2}=\frac1{(v'+1)(x_1-u)+u-v},
$$
and so
$$
t_2=\frac{m'}{v'+1}+\frac{m'(v-u)/(v'+1)-m+f'(u)}{(v'+1)(x_1-u)+u-v},
$$
which means
\eq[e4]{
t_2=\frac{m'}{v'+1},\quad m'(v-u)=\bigl(m-f'(u)\bigr)(v'+1).
}
The last equation, together with~\eqref{e3} and a bit of algebra, gives
$$
\frac{f(v)-f(u)}{v-u}\;v'=\frac{f'(v)+f'(u)}2\;v'.
$$

If $v'\ne0,$ then
$$
\frac{f(v)-f(u)}{v-u}=\frac{f'(v)+f'(u)}2,
$$
which, for $f(s)=|s|^p,$ is possible only if $v=-u,$ i.e.~only in the
just-considered case of $L_0.$

Therefore, $v'(u)=0$ and so the trajectories enter through different
points for different values of $u,$ but exit through the same point
$(v,v^2).$ Again, since we want the line $x_2=2ax_1+\ve^2-a^2$ as a
trajectory, we should have the exit point in one of the corners of
$\Omega_L(a),$ i.e. either $v=a-\ve=u_-$ or $v=a+\ve=u_+.$ In fact,
each of these choices yields a solution that is either concave or
convex in the whole region $\Omega_L(a).$
\begin{figure}[ht]
\begin{picture}(150,100)
\thicklines
\qbezier(-10,60)(75,-30)(160,60)
\qbezier(-10,85)(75,-5)(160,85)
\qbezier(25,31)(25,31)(150,50)
\put(85,45){$\scriptstyle a$}
\put(-2,22){$\scriptstyle u_-=a-\ve$}
\put(85,41){\circle*{2}}
\put(25,31){\circle*{2}}
\put(150,50){\circle*{2}}
\thinlines
\qbezier(25,31)(83,36)(141,42)
\qbezier(25,31)(72,32)(130,33)
\qbezier(25,31)(70,28)(115,25)
\qbezier(25,31)(61,25)(98,19)
\qbezier(25,31)(51,23)(78,15)
\end{picture}
\caption{The region $\Omega_L(a)$ foliated according to $L^-$}
\label{fig 0.01}
\end{figure}
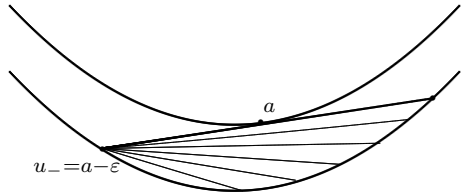
\begin{figure}[ht]
\begin{picture}(150,100)
\thicklines
\qbezier(-10,60)(75,-30)(160,60)
\qbezier(-10,85)(75,-5)(160,85)
\qbezier(25,31)(25,31)(150,50)
\put(85,45){$\scriptstyle a$}
\put(153,47){$\scriptstyle u_+=a+\ve$}
\put(85,41){\circle*{2}}
\put(25,31){\circle*{2}}
\put(150,50){\circle*{2}}
\thinlines
\qbezier(150,50)(150,50)(32,26)
\qbezier(150,50)(150,50)(42,21)
\qbezier(150,50)(150,50)(55,17)
\qbezier(150,50)(150,50)(71,15)
\qbezier(150,50)(150,50)(90,16)
\end{picture}
\caption{The region $\Omega_L(a)$ foliated according to $L^+$}
\label{fig 0.02}
\end{figure}
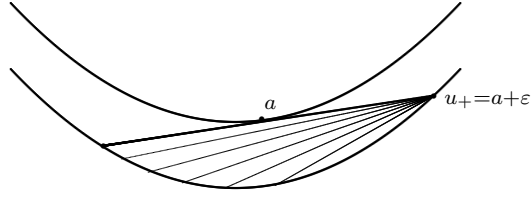

To determine whether the solution $L$ is concave or convex for each
choice of $v,$ we compute, using \eqref{e4} with $v'=0,$
$$
L_{x_2x_2}=m''(u)u_{x_2}=\frac{m''(u)}{x_1-v}.
$$
Thus $\sgn(L_{x_2x_2})=\sgn[m''(u)/(x_1-v)].$ A simple calculation
shows that $m''(u)=\frac13f'''(\xi)$ for some $\xi$ between $u$ and
$v.$ For $v=u_-$ we have $x_1\ge v,$ while for $v=u_+$ we have
$x_1\le v.$ Thus we obtain two solutions in $\Omega_L(a):$
\eq[g21]{
L^\pm(x;a)=\frac{f(u_\pm)-f(u)}{u_\pm-u}(x_1-u_\pm)+f(u_\pm),
}
with $u_\pm$ and $u$ given by
\eq[g20.0]{
u_\pm=a\pm\ve,\qquad u=\frac{x_2-u_\pm x_1}{x_1-u_\pm},
}
and
\eq[g20.1]{
\sgn(L^-_{x_2x_2})=\sgn(f'''),\quad \sgn(L^+_{x_2x_2})=-\sgn(f''').
}
The last two identities make sense if $f'''$ does not change its
sign on the interval $(a-\ve,a+\ve),$ as is the case for $f(s)=|s|^p.$
\subsection{The region between two tangents}
\label{t}
Here we fix $u$ and consider the region $\Omega_T(u)$ between the
two differently directed one-sided tangents sharing a common
lower-boundary point $(u,u^2)$ (see Figure~\ref{fig 0.1}):
$$
\Omega_T(u)=\{u-\ve\le x_1\le u+\ve\,,\;2ux_1-u^2+2\ve|u-x_1|\le
x_2\le x_1^2+\ve^2\}.
$$
As before, symmetry considerations dictate that for each $\Omega_T(u)$
we must either have $u=0$ or $\Omega_T(|u|)\subset\Oe^+$ (that is
$|u|\ge\ve$). In either case, the need for a solution --- let us call
it $T=T(x;u)$ --- in this region arises when we have to glue two
foliations (typically, those for $F^-$ and $F^+$ or $L^-$ and $L^+$).
This allows us to use the optimality principle: since each of the two
tangents bounding $\Omega_T(u)$ is an element of a Bellman foliation
in a portion of $\Oe,$ the compound Bellman candidate being built is
linear along each tangent. Since we are to construct either the smallest
concave or the largest convex \ma solution in $\Omega_T(u),$ $T$ is a
linear function of $x.$

Let us first briefly consider the case $u=0.$ We are looking for
a function
$$
T_0(x)=\alpha_1x_1+\alpha_2x_2+\alpha_0
$$
on $\Omega_T(0).$ As discussed earlier, $T_0$ is a function of $x_2$
only and so $\alpha_1=0.$ In addition $T_0(0,0)=f(0),$ which gives
\eq[t1]{
T_0(x)=\alpha x_2+f(0).
}
The constant $\alpha$ is determined by reading the boundary value off
the tangent $x_2=2\ve x_1;$ that value, in turn, depends on the other
components of the global Bellman foliation. We will see in the next
section how this simple step is accomplished.

The situation when $u\ne0$ is more involved. In theory, different
choices of $f$ may imply the need to use $\Omega_T$ to glue various
combinations of foliations of $\Omega_{F^\pm}$ and $\Omega_{L^\pm},$
described in the earlier sections. Thus, in general we are looking
for a linear candidate $T$ in the form
\eq[t2.01]{
T=\alpha_1x_1+\alpha_2x_2+\alpha_0,
}
such that it is equal to a particular candidate, $G^-,$ along its left
bounding tangent, the line $x_2=2(u-\ve)x_1-u^2+2\ve u,$ and to another
candidate, $G^+,$ along its right bounding tangent,
$x_2=2(u+\ve)x_1-u^2-2\ve u.$ That is we want to ensure that 
\eq[t2.015]{
T(x^\pm;u)=G^\pm(x^\pm),
}
where $x^-$ is any point on the left tangent and $x^-$ is any point on the right one.

The left and right tangents are assumed
to be extremal trajectories for $G^-$ and $G^+,$ respectively, and so
each function is linear along the appropriate tangent. We already have equality at the corner $(u,u^2),$ so it is
sufficient to glue our solutions at the other two corners of $\Omega_T(u):$
\eq[t2.02]{
\begin{aligned}
G^-(u-\ve,u^2-2\ve u+2\ve^2)
&=\alpha_1(u-\ve)+\alpha_2(u^2-2\ve u+2\ve^2)+\alpha_0,
\\
G^+(u+\ve,u^2+2\ve u+2\ve^2)
&=\alpha_1(u+\ve)+\alpha_2(u^2+2\ve u+2\ve^2)+\alpha_0.
\end{aligned}
}

In this paper, the only situation where we encounter $\Omega_T(u)$ with
$u\ne0,$ is when it is used as a transition regime between $\Omega_{F^+}(u_1,u)$
and $\Omega_{F^-}(u,u_2)$ for some numbers $u_1,u_2.$ It turns out that
such a transition places an important restriction on $u.$ Let us elaborate.
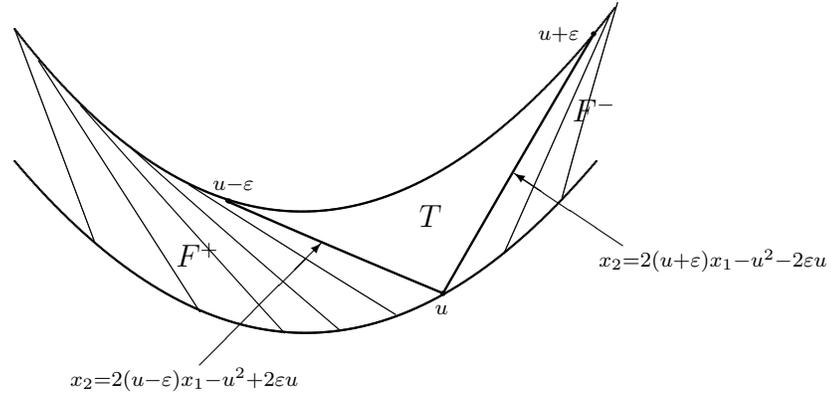
\begin{figure}[ht]
\begin{picture}(200,150)
\thicklines
\qbezier(-10,90)(100,-40)(210,90)
\qbezier(-10,140)(100,-2)(217,148)
\qbezier(70,75)(111,57)(152,40)
\qbezier(152,40)(180,89)(209,138)
\put(10,5){$\scriptstyle x_2=2(u-\ve)x_1-u^2+2\ve u$}
\put(210,50){$\scriptstyle x_2=2(u+\ve)x_1-u^2-2\ve u$}
\put(64,79){$\scriptstyle u-\ve$}
\put(148,32){$\scriptstyle u$}
\put(187,136){$\scriptstyle u+\ve$}
\put(50,50){\large $F^+$}
\put(200,105){\large $F^-$}
\put(142,65){\large $T$}
\thinlines
\put(60,13){\vector(1,1){46}}
\put(220,58){\vector(-3,2){41}}
\put(71,75){\circle*{2}}
\put(152,40){\circle*{2}}
\put(209,138){\circle*{2}}
\thinlines
\qbezier(56,81)(95,56)(134,32)
\qbezier(34,95)(73,60)(113,26)
\qbezier(15,111)(53,68)(92,25)
\qbezier(-1,128)(30,80)(60,33)
\qbezier(-10,140)(5,100)(20,60)
\qbezier(175,55)(195,100)(215,145)
\qbezier(197,76)(207,113)(218,150)
\end{picture}
\caption{The region $\Omega_T(u)$ connecting the
foliations for $F^+$ and $F^-.$}
\label{fig 0.1}
\end{figure}

We are looking for a solution in the form~\eqref{t2.01} according
to two requirements. The first one is that
relations~\eqref{t2.02} be fulfilled, with $G^-=F^-$ and $G^+=F^+$.
The second requirement is that the solution must preserve the sign
of the (generalized) second derivative along any direction. This second
requirement is, in fact, the reason why $T$ is necessary: both
$F^+$ and $F^-$ are continuous in the whole domain, but their various
second derivatives change signs.

Recall that we can write $F^+$ and $F^-$ as
$$
F^+=m_+(u)(x_1-u)+f(u),\quad F^-=m_-(u)(x_1-u)+f(u),
$$
where the coefficients $m_+$ and $m_-$ satisfy the differential
equations~\eqref{ff009} and~\eqref{ff010}, respectively:
\eq[t2]{
m_+'=-\frac1\ve(m_+-f'),\quad m_-'=\frac1\ve(m_--f').
}

For $G^-=F^+$ and $G^+=F^-,$ \eqref{t2.02} gives
\begin{align*}
m_+(u)(-\ve)+f(u)&=\alpha_1(u-\ve)+\alpha_2(u^2-2\ve u+2\ve^2)+\alpha_0,
\\
m_-(u)\ve+f(u)&=\alpha_1(u+\ve)+\alpha_2(u^2+2\ve u+2\ve^2)+\alpha_0.
\end{align*}
Together with the boundary condition
$$
f(u)=\alpha_1u+\alpha_2u^2+\alpha_0
$$
this yields
\begin{align}
\notag \alpha_1&=\frac{m_++m_-}2-\frac{m_--m_+}{2\ve}\,u,
\\
\label {al1} \alpha_2&=\frac{m_--m_+}{4\ve},
\\
\notag \alpha_0&=-\frac{m_++m_-}2\,u+\frac{m_--m_+}{4\ve}\,u^2+f.
\end{align}

This determines the function $T$ up to the parameter $u,$ which
is {\it not} free. It has to be chosen so that the Bellman candidate
$G$ constructed of the blocks $F^+,$ $T,$ and $F^-$ is either concave
or convex in the combined domain
$\Omega_{F^+}(u_1,u)\cup\Omega_T(u)\cup\Omega_{F^-}(u,u_2).$ We cannot
enforce differentiability along the tangents bounding $\Omega_T(u),$
but can, and will, require that the jumps in the first derivative(s) be
of the same sign on both boundaries. We will see in Section~\ref{induction} that it is enough to check this
in a single direction transversal to the boundary.
Let us do so for $G_{x_2}.$
Along the left and right bounding tangents the jump in this derivative
is, respectively, $\alpha_2-F^+_{x_2}$ and $\alpha_2-F^-_{x_2}.$ Recall
that each $F^+_{x_2}$ and $F^-_{x_2}$ has a constant value along the
corresponding line, given, respectively, by
$$
F^+_{x_2}=-\frac1{2\ve}(m_+-f')\quad\text{and}
\quad F^-_{x_2}=\frac1{2\ve}(m_--f'),
$$
where we have used~\eqref{ff008} and~\eqref{t2}. Taking into account
formula \eqref{al1} for $\alpha_2,$  we can write the compatibility
condition as
$$
\left[\frac{m_--m_+}{4\ve}+\frac1{2\ve}(m_+-f')\right]
\left[\frac{m_--m_+}{4\ve}-\frac1{2\ve}(m_--f')\right]
=-\frac1{16\ve^2}\left[m_++m_--2f'\right]^2\ge0,
$$
and, finally, as
\eq[t3]{
m_+(u)+m_-(u)=2f'(u).
}

We see that, indeed, in gluing $F^+$  and $F^-$ via $T,$ the parameter
$u$ cannot be chosen arbitrarily. In fact, only in certain circumstances
such a value $u$ exists. If it does, we can write
the function $T$ using \eqref{t2.01}, \eqref{al1}, and \eqref{t3} as
\eq[t4]{
T(x;u)=f'(u)(x_1-u)+\frac{m_-(u)-f'(u)}{2\ve}\,(-2x_1u+x_2+u^2)+f(u).
}

This completes the description of the four basic blocks out of which
we will assemble global Bellman candidates. The next section is devoted to this task.
\section{Global Bellman candidates}
\label{global}
In this section, we construct a set of global Bellman candidates, i.e. candidates that have the same sign of the generalized second
differential in the whole domain $\Oe.$ The main emphasis is on our specific choice of the embedding function $f,$ $f(s)=|s|^p,$
although some results are stated in more generality, which will be useful a bit later.

Although the global Bellman foliations do significantly depend on the range
of $p$ considered, some useful statements can be made for all $p.$ In
the previous section, we built general candidates $F^+(x;u_1,u_2)$ and
$F^-(x;u_1,u_2),$ each determined up to a constant $C.$ To specify this
constant for $F^+(x;u_1,u_2),$ we need to know its left neighbor; for
$F^-(x;u_1,u_2),$ we need to know its right neighbor. If $F^+$ has no
left neighbor, i.e. if~$u_1=-\infty,$ the constant is determined
using limiting considerations, and similarly for $F^-$ in the case
$u_2=\infty.$

From our limited arsenal of canonical blocks, only blocks of the $L$
type can be located directly to the left of $F^+(x;u_1,u_2).$ This gives us
a boundary condition for $F^+$ on the line $x_2=2(u_1-\ve)x_1-u_1^2+2u_1\ve,$ shared by the two canonical
sub-domains.  We know from~\eqref{g21} and~\eqref{g20.0} that
$$
L(a,a^2+\ve^2;a)=\frac{f(a-\ve)+f(a+\ve)}2\,.
$$
Setting $a=u_1-\ve$ and equating the result with $F^+(x;u_1,u_2)$ from~\eqref{ff4.0} for $u=u_1$ and $x_1=u_1-\ve,$
we obtain
$$
C=\frac{f(u_1)-f(u_1-2\ve)}{2\ve}e^{u_1/\ve}\,,
$$
which gives
\eq[gg4.0]{
F^+(x;u_1,u_2)=\frac1\ve e^{-u/\ve}\left[\frac{f(u_1)-f(u_1-2\ve)}2
e^{u_1/\ve}+\int_{u_1}^uf'(s)e^{s/\ve}\,ds\right]\, (x_1-u)+f(u),
}
where $u$ is given as a function of $x$ by~\eqref{ff5}.

Let us specify this formula for $f(s)=|s|^p$ in two cases we will need:
$u_1=\ve$ and $u_1=2\ve$.
In the first case we have, after a change of variable in the integral,
\eq[h1]{
F^+(x;\ve,u_2)= p\ve^{p-1}e^{-u/\ve}
\left[\int_1^{u/\ve}\!\!\!t^{p-1}e^t\,dt\right]
(x_1-u)+u^p,\quad x\in\Omega_{F^+}(\ve,u_2),
}
and in the second one,
\eq[hh1]{
F^+(x;2\ve,u_2)=\ve^{p-1}e^{-u/\ve}
\left[2^{p-1}e^2+p\int_2^{u/\ve}\!\!\!t^{p-1}e^t\,dt\right]
(x_1-u)+u^p,\quad x\in\Omega_{F^+}(2\ve,u_2),
}
where
$$
u=x_1+\ve-\sqrt{\ve^2-x_2+x_1^2}\,.
$$

We then have $\sgn(F^+_{x_2x_2})=\sgn(\tau_+),$ where, according
to~\eqref{ff36},
\eq[h2.5]{
\tau_+(u)=p(p-2)\ve^{p-2}
\left[e+(p-1)\int_1^{u/\ve}\!\!\!t^{p-3}e^t\,dt\,\right]\,,
}
if $u_1=\ve,$ and
\eq[hh27.6]{
\tau_+(u)=(p-1)(p-2)\ve^{p-2}
\left[2^{p-2}e^2+p\int_2^{u/\ve}\!\!\!t^{p-3}e^t\,dt\,\right]\,,
}
if $u_1=2\ve.$

Now, let us consider the case when the block $\Omega_{F^+}$ has no
left neighbor, i.e.~when $u_1=-\infty.$ The argument here is subtler.
Namely, we first look at the sign of $F^+_{x_2x_2}$ and then invoke
the optimality principle. For large negative $u$ we have
$\sgn(F^+_{x_2x_2})=\sgn(\tau_+)=\sgn(C),$ with $\tau_+$ given
by~\eqref{ff36}. Thus, if $C<0,$ we have an upper Bellman candidate.
Since we want {\it the smallest} upper candidate, and $F^+$ given
by~\eqref{ff4.0} decreases as $C$ grows, we set $C=0.$ On the other hand,
if $C>0,$ we have a lower candidate, which we want to maximize; this,
again, leads us to take $C=0.$ By symmetry, $\Omega_{F^+}(-\infty,u_2)$ must be contained in $\Oe\cap\{x_1\le0\}$ and so we can restrict ourselves to $u_2\le0.$ Therefore, we obtain the following
expression:
\eq[h2.95]{
F^+(x;-\infty,u_2)=p\ve^{p-1}e^{-u/\ve}
\left[\int^{\infty}_{-u/\ve}\!\!\!\!t^{p-1}e^{-t}\,dt\right]
(u-x_1)+|u|^p,\quad x\in\Omega_{F^+}(-\infty,u_2),
\quad \!\!u\le u_2\le0,
}
with $u$ given by
$$
u=x_1+\ve-\sqrt{\ve^2-x_2+x_1^2}\,.
$$

The consideration for $F^-(x;u_1,u_2)$ is entirely symmetrical:
\eqref{gg4.0}, \eqref{h1}, \eqref{hh1}, and~\eqref{h2.95} become,
respectively,
\eq[gg4.01]{
F^-(x;u_1,u_2)=\frac1\ve e^{u/\ve}
\left[\frac{f(u_2+2\ve)-f(u_2)}2 e^{-u_2/\ve}+
\int_u^{u_2}\!\!f'(s)e^{-s/\ve}\,ds\right]\, (x_1-u)+f(u),
}
$$
F^-(x;u_1,-\ve)= p\ve^{p-1}e^{u/\ve}
\left[\int_1^{-u/\ve}\!\!\!t^{p-1}e^t\,dt\right]
(u-x_1)+|u|^p,\quad x\in\Omega_{F^-}(u_1,-\ve),
$$
$$
F^-(x;u_1,-2\ve)= \ve^{p-1}e^{-u/\ve}
\left[2^{p-1}e^2+p\int_2^{-u/\ve}\!\!\!t^{p-1}e^t\,dt\right]
(u-x_1)+|u|^p,\quad x\in\Omega_{F^-}(u_1,-2\ve),
$$
and
\eq[h3]{
F^-(x;u_1,\infty)= p\ve^{p-1}e^{u/\ve}
\left[\int^\infty_{u/\ve}\!\!\!t^{p-1}e^{-t}\,dt\right]
(x_1-u)+u^p,\quad x\in\Omega_{F^-}(u_1,\infty),\quad 0\le u_1\le u,
}
with $u$ given by
$$
u=x_1-\ve+\sqrt{\ve^2-x_2+x_1^2}\,.
$$
We will also need $\sgn(F^-_{x_2x_2}(x;u_1,\infty))=\sgn(\tau_-),$ where,
according to~\eqref{ff37},
\eq[h4.5]{
\tau_-(u)=-p(p-1)(p-2)\ve^{p-2}\int_{u/\ve}^\infty\!\!t^{p-3}e^{-t}\,dt.
}

We will now build global Bellman candidates starting ``from the middle'':
the only canonical sub-domains that can symmetrically incorporate the
line $x_1=0$ are $\Omega_L(0)$ and $\Omega_T(0).$ We fix one of these
and glue other canonical sub-domains to it, so as to preserve the global
convexity/concavity of the resulting candidate. It is convenient to
split further discussion in two parts: $p\ge1$ and $0<p<1.$

Before we proceed, let us fix the following notation for the
coefficients of $x_1$ in~\eqref{gg4.0} and~\eqref{gg4.01} for $f(s)=|s|^p:$
\begin{align}
\label{h4.1}
m_+(u;u_1)
&=e^{-u/\ve}\left[\frac{u_1^p-|u_1-2\ve|^p}{2\ve}e^{u_1/\ve}
+p\ve^{p-1}\int_{u_1/\ve}^{u/\ve}\!\!t^{p-1}e^t\,dt\right],
&0\le u_1\le u,
\\
\label{h4.2}
m_-(u;u_2)
&=e^{u/\ve}\left[\frac{u_2^p-(u_2+2\ve)^p}{2\ve}e^{-u_2/\ve}
+p\ve^{p-1}\int_{u/\ve}^{u_2/\ve}\!\!t^{p-1}e^{-t}\,dt\right],
&0\le u\le u_2.
\end{align}
In addition, although most of our work here is with the power function, in Section~\ref{other} we will need the more general counterparts of \eqref{h4.1} and \eqref{h4.2} in the specific cases when $u_1=\ve$ and $u_2=\infty,$ respectively:
\eq[h4.3]{
m^f_+(u)=e^{-u/\ve}\int_1^{u/\ve}\!\!\!f'(\ve t)e^t\,dt,
}
\eq[h4.4]{
m^f_-(u)=e^{u/\ve}\int_{u/\ve}^\infty\!f'(\ve t)e^{-t}\,dt.
}

\subsection{The case $p\ge 1$}
Let us first consider the split $\Oe=\Omega_{F^-}(-\infty,-\ve)\cup
\Omega_L(0)\cup\Omega_{F^+}(\ve,\infty).$ According to~\eqref{g1},
the solution in $\Omega_L(0)$ is given by
$$
G(x)=L_0(x)=x_2^{p/2},
$$
hence, in that region we have $G_{x_2x_2}(x)=\frac14p(p-2)x^{p/2-2}$ and
so $\sgn(G_{x_2x_2})=\sgn(p-2).$ In $\Omega_{F^+}(\ve,\infty),$
by~\eqref{h2.5} we have $\sgn(F^+_{x_2x_2})=\sgn(p-2)$ for $p\ge1.$
We, thus, attempt to check whether setting $G(x)=F^+(x;\ve,\infty)$ in
$\Omega_{F^+}(\ve,\infty)$ will produce an acceptable Bellman candidate.
Along the line $x_2=\ve^2$ (the shared boundary of the two sub-domains),
we have $L_0=\ve^p,$ while~\eqref{h1} gives $F^+=\ve^p$ for $u=\ve.$ 
In Section~\ref{induction}, we will verify the convexity/concavity of the resulting candidate
in the combined domain $\Omega_L(0)\cup\Omega_{F^+}(\ve,\infty)$
Subject to that verification, we have
a complete candidate in the $\Oe^+$ and hence, by symmetry, in the
whole $\Oe:$
\eq[ff5.9]{
M(x)=
\begin{cases}
F^-(x;-\infty,-\ve),&x\in\Omega_{F^-}(-\infty,-\ve),\\
L_0(x),&x\in\Omega_L(0),\\
F^+(x;\ve,\infty),&x\in\Omega_{F^+}(\ve,\infty).
\end{cases}
}
More explicitly,
\eq[ff6]{
M_{\ve,p}(x)=
\begin{cases}
\ds m_+(u;\ve)(|x_1|-u)+u^p,
&x\in\Omega_{F^-}(-\infty,-\ve)\cup\Omega_{F^+}(\ve,\infty),
\rule[-10pt]{0pt}{10pt}
\\
x_2^{p/2},&x\in\Omega_L(0),
\end{cases}
}
with
$$
u=|x_1|+\ve-\sqrt{\ve^2-x_2+x_1^2}.
$$
The corresponding foliation of $\Oe$ is shown on Figure~\ref{fig1}.
\begin{figure}[ht]
\begin{picture}(200,150)
\put(0,10){\vector(1,0){200}}
\put(100,0){\vector(0,1){150}}
\thicklines
\qbezier(-40,120)(100,-100)(240,120)
\qbezier(-10,140)(100,0)(210,140)
\qbezier(-3,70)(0,70)(203,70)
\thinlines
\qbezier(5,60)(100,60)(195,60)
\qbezier(16,50)(100,50)(184,50)
\qbezier(27,40)(100,40)(173,40)
\qbezier(40,30)(100,30)(160,30)
\qbezier(58,20)(100,20)(142,20)
\qbezier(95,10)(100,10)(105,10)
\qbezier(105,70)(105,70)(211,80)
\qbezier(127,74)(127,74)(226,100)
\qbezier(145,81)(145,81)(245,135)
\qbezier(95,70)(95,70)(-11,80)
\qbezier(73,74)(73,74)(-26,100)
\qbezier(55,81)(55,81)(-45,135)
\put(95,35){$\Omega_L$}
\put(180,93){$\Omega_{F^+}$}
\put(-5,93){$\Omega_{F^-}$}
\end{picture}
\caption{The Bellman foliation for the candidate $M.$}
\label{fig1}
\end{figure}
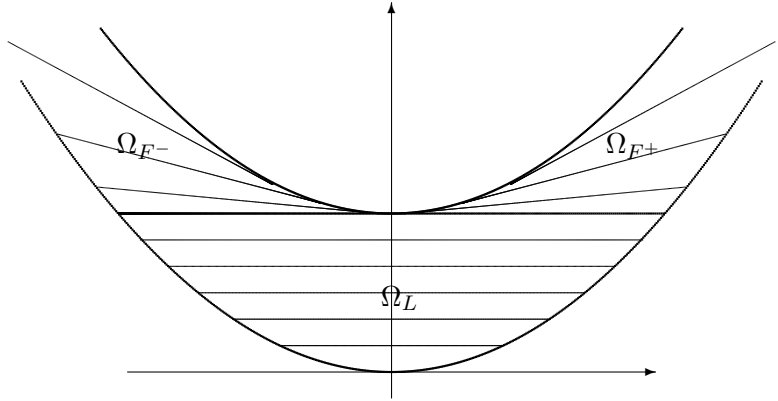
This function gives an upper Bellman candidate for $1\le p\le 2$ and a
lower one for $2\le p<\infty.$ Let us write separately the candidate
for the important case $p=1,$ when the integrals can be evaluated
explicitly:
\eq[ff7]{
M_{\ve,1}(x)=
\begin{cases}
\ds |x_1|+\left(\ve-\sqrt{\ve^2-x_2+x_1^2}\right)
&\ds\!\!\!\!\!\exp\frac{-|x_1|+\sqrt{\ve^2-x_2+x_1^2}}\ve\,,\\
&x\in\Omega_{F^-}(-\infty,0)\cup\Omega_{F^+}(0,\infty),\\
x_2^{1/2},&x\in\Omega_L(0).\rule{0pt}{20pt}
\end{cases}
}

To get the other Bellman candidate, we consider the split
$\Oe=\Omega_{F^+}(-\infty,0)\cup\Omega_{T}(0)\cup\Omega_{F^-}(0,\infty).$
In $\Omega_{F^-}(0,\infty)$ we, naturally, set
$$
G(x)=F^-(x;0,\infty).
$$
According to~\eqref{h4.5}, $\sgn(F^-_{x_2x_2})\!=\!-\sgn(p-2).$
On the shared boundary of $\Omega_T(0)\!$ and $\Omega_{F^-}(0,\infty),$
the line $x_2=2\ve x_1,$ \eqref{h3} gives
$$
F^-|_{u=0}=p\ve^{p-1}\left[\int_0^\infty\!\!t^pe^{-t}dt\right]x_1
=p\ve^{p-1}\Gamma(p)\,x_1=\frac p2\ve^{p-2}\Gamma(p)x_2.
$$

In $\Omega_T(0),$ we set
$$
G(x)=T_0(x),
$$
where, from~\eqref{t1},
$$
T_0(x)=\alpha x_2.
$$
To preserve continuity along the line $x_2=2\ve x_1,$ we set
$\alpha=\frac p2\ve^{p-2}\Gamma(p).$ Again, we postpone until the next section the verification
that the resulting candidate is locally convex/concave in $\Omega_T(0)\cup\Omega_{F^-}(0,\infty).$
By symmetry,
we obtain the following global candidate:
\eq[m2.9]{
N(x)=
\begin{cases}
F^+(x;-\infty,0),&x\in\Omega_{F^+}(-\infty,0),\\
T_0(x),&x\in\Omega_T(0),\\
F^-(x;0,\infty),&x\in\Omega_{F^-}(0,\infty).
\end{cases}
}
More specifically,
\eq[m3]{
N_{\ve,p}(x)=
\begin{cases}
\ds m_-(u;\infty)(|x_1|-u)+u^p,
&x\in\Omega_{F^+}(-\infty,0)\cup\Omega_{F^-}(0,\infty),
\rule[-10pt]{0pt}{10pt}
\\
~\ds\frac p2\,\ve^{p-2}\Gamma(p)\,x_2,&x\in\Omega_T(0),
\end{cases}
}
where
$$
u=|x_1|-\ve+\sqrt{\ve^2-x_2+x_1^2}\,.
$$
This function gives an upper candidate for $p\ge2$ and a lower one
for $1\le p\le 2.$ The corresponding Bellman foliation is shown on
Figure~\ref{fig2}. We note that this geometric description is accurate
for all $p>1,$ but not for $p=1.$ In that case, we can again evaluate
the integrals explicitly and thus obtain
\eq[m6]{
N_{\ve,1}(x)=
\begin{cases}
\ds |x_1|,&x\in\Omega_{F^+}(-\infty,0)\cup\Omega_{F^-}(0,\infty),
\rule[-10pt]{0pt}{10pt}
\\
~\ds\frac 1{2\ve}\,x_2,&x\in\Omega_T(0).
\end{cases}
}
Therefore, $N_{\ve,1}$ is a piecewise linear function and the defect of
its Hessian is 2 in the interior of each canonical sub-domain involved.
Thus every straight line lying entirely in $\Omega_{F^-}(0,\infty),$
$\Omega_{F^+}(-\infty,0),$ or $\Omega_T(0)$ is an extremal trajectory
for $N_{\ve,1}.$

\begin{figure}[ht]
\begin{picture}(200,150)
\put(0,10){\vector(1,0){200}}
\put(100,0){\vector(0,1){150}}
\thicklines
\qbezier(0,65)(100,-45)(200,65)
\qbezier(-10,140)(100,0)(210,140)
\qbezier(100,10)(100,10)(185,111)
\qbezier(100,10)(100,10)(15,111)
\thinlines
\qbezier(118,12)(118,12)(196,122)
\qbezier(132,15)(132,15)(203,131)
\qbezier(148,23)(148,23)(209,138)
\qbezier(173,40)(173,40)(211,140)
\qbezier(82,12)(82,12)(4,122)
\qbezier(68,15)(68,15)(-3,131)
\qbezier(52,23)(52,23)(-9,138)
\qbezier(27,40)(27,40)(-11,140)
\put(95,50){$\Omega_T$}
\put(168,75){$\Omega_{F^-}$}
\put(15,75){$\Omega_{F^+}$}
\end{picture}
\caption{The Bellman foliation for the candidate $N.$}
\label{fig2}
\end{figure}
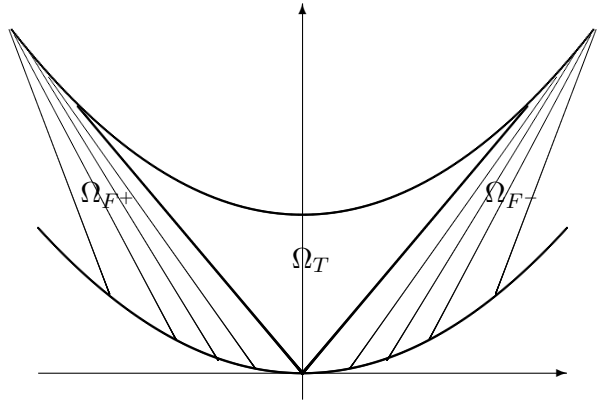
\subsection{The case $0<p<1$}
In this case, we have two Bellman candidates of a more complicated
nature. As before, we build each global solution starting with
either $L_0(x)$ or $T_0(x)$ (thus placing either $\Omega_L(0)$ or
$\Omega_T(0)$ at the center of $\Oe$) and then extending the solution
appropriately to the whole $\Oe.$

Let us first set $G(x)=L_0(x)=x_2^{p/2}$ in
$\Omega_L(0).$ Since $\sgn(L_{x_2x_2})=\sgn(p(p-2))<0,$ we are building
an upper Bellman candidate. It is natural to attempt to glue a solution
$F^+(x;\ve,u_2)$ to this foundation, for some $u_2>0,$ just as we did
in the previous case. Since $\sgn(F^+_{x_2x_2})=\sgn(\tau_+),$ and,
by~\eqref{h2.5}, $\tau_+(\ve)<0,$ this is a proper choice in that we
still have a concave candidate. However, we cannot take $u_2=\infty,$
as we did before, since $\tau_+(u)>0$ for sufficiently large $u.$ This means
that the only canonical solution we can have after all transient effects
have dissipated (i.e. for large $u,$ after all necessary transition
regimes have been deployed) is $F^-(x;u_1,\infty).$ Again, we check
the sign of $F^-_{x_2x_2}:$ $\sgn(F^-_{x_2x_2})=\sgn(\tau_-)$ and,
by~\eqref{h4.5}, $\tau_-(u)<0$ for all $u.$ Therefore, we need a
transition regime connecting the foliations for $F^+(x;\ve,u_1)$ and
$F^-(x;u_1,\infty).$ We have an obvious choice, one that was considered
in section~\ref{t}: $T(x;\xi)$ for a specific value of $\xi.$
Condition~\eqref{t3} dictates that $T(x;\xi)$ would appropriately
glue $F^+(x;\ve,\xi)$ and $F^-(x;\xi,\infty)$ if and only if
\eq[h7.1]{
m_+(\xi)+m_-(\xi)=2p\,\xi^{p-1}.
}
Slightly rewriting the integrals in~\eqref{h4.1} and~\eqref{h4.2},
and letting $\xi=\mu\ve,$ we can reformulate this condition as follows:
\eq[h8]{
e^{-\mu}\int_1^\mu\!\!z^{p-1}e^z\,dz
+e^{\mu}\int_\mu^\infty\!\!\! z^{p-1}e^{-z}\,dz=2\mu^{p-1}.
}
Obviously, we need to have $\mu>1,$ that is $\xi-\ve=\ve(\mu-1)>0.$
We must verify that such a $\mu$ exists and that $\tau_+(\xi)\le0,$
meaning $F^+$ remains an upper candidate up to the line $u=\xi.$
Changing the variable in~\eqref{h2.5} and integrating by parts twice,
we rewrite this condition as
\eq[h9]{
e^{-\mu}\int_1^\mu\!\!z^{p-1}e^z\,dz\le\mu^{p-1}+(1-p)\mu^{p-2}.
}
We are about to prove the existence of the solution of~\eqref{h8}
satisfying~\eqref{h9}. We prove a slightly more general result, which we will need in Secion~\ref{other}. 
Namely, let us replace the condition \eqref{h7.1} with
$$
m^f_+(\xi)+m^f_-(\xi)=2f'(\xi),
$$
where the functions $m^f_\pm$ are defined by~\eqref{h4.3}
and~\eqref{h4.4}. Rewriting the integrals, we get
\eq[h9.2]{
e^{-\mu}\int_1^\mu\!\!f'(z\ve)e^z\,dz+
e^\mu\int_\mu^\infty\!\!\!f'(z\ve)e^{-z}\,dz=2f'(\mu\ve).
}
In addition, recall that for a general $f$ we have $\sgn(F^+_{x_2x_2})
=\sgn(\tau_+^f),$ where we have set, using $u_1=\ve$ and $C=0$
in~\eqref{ff35.9},
$$
\tau_+^f(u)=(f''(u)-f'(u)/\ve)e^{u/\ve}+
\frac1{\ve^2}\int_\ve^u\!\!\!f'(s)e^{s/\ve}\,ds.
$$
After rewriting, we get $\sgn(\tau_+^f(\xi))=
\sgn(\tau_+^f(\ve\mu)=\sgn(g(\mu)),$ where
$$
g(\mu)\df e^{-\mu}\int_1^\mu\!\! f'(z\ve)e^z\,dz-f'(\mu\ve)+
\ve f''(\mu\ve).
$$

\begin{lemma}
\label{lm}
Fix $\ve>0$ and let $f$ be a thrice-differentiable function on
$(\ve,\infty)$ satisfying
\eq[a2]{
\lim_{s\to\infty}e^s|f'''(s\ve)|=\infty
}
and either
$$
\hbox{Case 1\textup:}\qquad f'(t)\ge0,\ f''(t)\le0,\ f'''(t)\ge0,
\quad \forall t\in(\ve,\infty),
$$
or
$$
\hbox{Case 2\textup:}\qquad f'(t)\le0,\ f''(t)\ge0,\ f'''(t)\le0,
\quad \forall t\in(\ve,\infty).
$$
Then\textup, for each $\ve>0,$ equation~\eqref{h9.2} has a unique
solution $\mu^*$ in the interval $(1,\infty).$ Furthermore, $\mu^*$
satisfies
\eq[a1]{
\begin{split}
g(\mu_*)\le0&\quad\hbox{in Case 1\textup,}\\
g(\mu_*)\ge0&\quad\hbox{in Case 2}.
\end{split}
}
\end{lemma}
\begin{proof}
Letting
$$
h(\mu)=\int_1^\mu\!\! f'(z\ve)e^z\,dz+
e^{2\mu}\int_\mu^\infty\!\! f'(z\ve)e^{-z}\,dz-2e^\mu f'(\mu\ve),
$$
we se that the task is to prove that $h$ has a unique zero $\mu^*$ in
$(1,\infty).$ We calculate:
$$
h'(\mu)=2e^{2\mu}\left[\int_\mu^\infty\!\! f'(z\ve)e^{-z}\,dz-
f'(\mu\ve)e^{-\mu}-\ve f''(\mu\ve)e^{-\mu}\right]=
2\ve^2e^{2\mu}\int_\mu^\infty\!\! f'''(z\ve)e^{-z}\,dz.
$$
Thus,
\begin{itemize}
\item[--]
Case 1: $h'(\mu)>0,$ $\forall\mu\ge1,$ and, from~\eqref{a2},
$\lim_{\mu\to\infty}h'(\mu)=\infty.$ In addition,
$$
h(1)=e^2\int_1^\infty\!\! f'(z\ve)e^{-z}\,dz-2ef'(\ve)=
\ve e^2\int_1^\infty\!\! f''(z\ve)e^{-z}\,dz-f'(\ve)e<0.
$$
\item[--]
Case 2: $h'(\mu)<0,$ $\forall\mu\ge1,$ $\lim_{\mu\to\infty}h'(\mu)=
-\infty,$ and $h(1)>0.$
\end{itemize}
In each case, this implies the existence of a unique root $\mu_*$ of $h.$

To finish the proof, observe that
$$
e^{-\mu_*}\int_1^{\mu_*}\!\!f'(z\ve)e^z\,dz=
2f'(\mu_*\ve)-e^{\mu_*}\int_{\mu_*}^\infty\!\! f'(z\ve)e^{-z}\,dz
$$
and so we have
$$
g(\mu_*)=-e^{\mu_*}\int_{\mu_*}^\infty\!\! f'(z\ve)e^{-z}\,dz+
f'(\mu_*\ve)+\ve f''(\mu_*\ve)=
-\ve e^{\mu_*}\int_{\mu_*}^\infty\!\! f'''(z\ve)e^{-z}\,dz,
$$
which yields \eqref{a1}.
\end{proof}

Setting $f(s)=|s|^p,$ we obtain an immediate
\begin{corollary}
\label{mu}
For each $p<1,$ equation~\eqref{h8} has a unique solution $\mu_*$ in
the interval $(1,\infty).$ Furthermore, $\mu_*$ satisfies~\eqref{h9}.
\end{corollary}
\begin{remark}
It is easy to show that $\mu_*(p)\to\infty,$ as $p\to1^-$ and
$\mu_*(p)\to1,$ as $p\to-\infty.$
\end{remark}
From now on, let us denote the solution of~\eqref{h8} simply by $\mu;$
also let $\xi=\mu\ve.$  The lemma just proved means that we have,
indeed, succeeded in building a complete Bellman candidate. On
$\Omega_T(\xi)$ that candidate is given by $T(x;\xi),$ where,
according to~\eqref{t4},
$$
T(x;\xi)=p\xi^{p-1}x_1+
\frac1{2\ve}(m_-(\xi;\infty)-p\xi^{p-1})(-2x_1\xi+x_2+\xi^2)+(1-p)\xi^p.
$$

Extending, as before, the solution to the left of the line $x_1=0$ by
symmetry, we can write down our global candidate:
\eq[h11]{
P(x)=
\begin{cases}
F^+(x;-\infty,-\xi),&x\in\Omega_{F^+}(-\infty,-\xi),\\
T(x;-\xi),&x\in\Omega_T(-\xi),\\
F^-(x;-\xi,-\ve),&x\in\Omega_{F^-}(-\xi,-\ve),\\
L_0(x),&x\in\Omega_L(0),\\
F^+(x;\ve,\xi),&x\in\Omega_{F^+}(\ve,\xi),\\
T(x;\xi),&x\in\Omega_T(\xi),\\
F^-(x;\xi,\infty),&x\in\Omega_{F^-}(\xi,\infty).
\end{cases}
}
This representation exhibits the geometric structure of $P.$
In addition, we need a usable formula:
\eq[h12]{
P_{\ve,p}(x)=
\begin{cases}
m_-(u_-;\infty)(|x_1|-u_-)+u_-^p,
&x\in\Omega_{F^+}(-\infty,-\xi)\cup\Omega_{F^-}(\xi,\infty),\\
p\xi^{p-1}|x_1|+\frac1{2\ve}(m_-(\xi;\infty)-p\xi^{p-1})
\rule[0pt]{0pt}{20pt}
(&\!\!\!\!\!\!-2|x_1|\xi+x_2+\xi^2)+(1-p)\xi^p,\\
&x\in\Omega_T(-\xi)\cup\Omega_T(\xi),\\
m_+(u_+;\ve)(|x_1|-u_+)+u_+^p,\rule[0pt]{0pt}{15pt}
&x\in\Omega_{F^-}(-\xi,-\ve)\cup\Omega_{F^+}(\ve,\xi),\\
x_2^{p/2},&x\in\Omega_L(0),\rule[0pt]{0pt}{20pt}
\end{cases}
}
where $\xi=\mu\ve,$ $\mu$ is the unique solution of~\eqref{h8} in
$(1,\infty),$ and
$$
u_+=|x_1|+\ve-\sqrt{\ve^2-x_2+x_1^2}\,,\quad
u_-=|x_1|-\ve+\sqrt{\ve^2-x_2+x_1^2}\,.
$$
As noted before, $P_{\ve,p}$ gives an upper Bellman candidate for $0<p<1.$
The foliation for this function is shown on Figure~\ref{fig67}.
\begin{figure}[ht]
  \centering{\includegraphics[width=11cm]{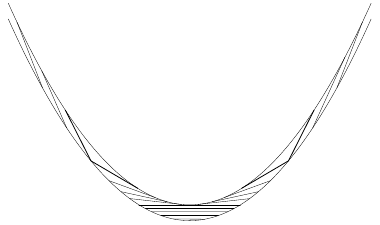}}
\caption{The Bellman foliation for the candidate $P.$}
\label{fig67}
\end{figure}

To construct the lower candidate, we place $\Omega_T(0)$ at the
center of $\Oe,$ thus setting
\eq[h26]{
G(x)=T_0(x)=\alpha x_2,\quad x\in \Omega_T(0),
}
with the constant $\alpha$ still to be determined. We now have to glue
a canonical candidate to the right of $T_0(x).$ Geometrically, we need
a candidate whose foliation includes the line $x_2=2\ve x_1$ and, thus,
have three choices: $F^-(x;0,\beta),$ for some $\beta;$ $L^+(x;\ve);$
and $L^-(x;\ve).$

Let us first examine $F^-(x;0,\beta).$ We have $\sgn(F^-_{x_2x_2})=
\sgn(\tau_-),$ where, according to~\eqref{ff16}, $\tau_-(u)=
p(p-1)O(u^{p-2})$ for small positive $u.$ This means that near
the boundary of $\Omega_T(0)$ (i.e. the line $u=0$), we have
$F^-_{x_2x_2}<0,$ that is $F^-$ gives an {\it upper} candidate,
while we are building a lower one.

For the other two possibilities, \eqref{g20.1} gives
$$
\sgn(L^\pm_{x_2x_2})=\mp\sgn(p(p-1)(p-2)),
$$
meaning only $L^-$ gives a lower candidate. Therefore, we set
$$
G(x)=L^-(x;\ve)=x_2^{p-1}x_1^{2-p},\quad x\in \Omega_L(\ve),
$$
where we have used the ``$-$'' part of formula \eqref{g21} with $u_-=0$ and $f(s)=|s|^p.$

We are now in a position to determine the constant $\alpha$
in~\eqref{h26}: setting $T_0(x_1,2\ve x_1)=L^-(x_1,2\ve x_1)$ gives
$$
\alpha=(2\ve)^{p-2}.
$$

Having determined our candidate in $\Omega_T(0)\cup\Omega_L(\ve),$ we
now have to glue another canonical solution to $L^-.$ Observe that for
sufficiently large $x_1$ we expect our candidate to be given by
$F^+(x;\gamma,\infty)$ (according to~\eqref{h4.5}, its counterpart, $F^-,$ determines an upper
candidate and so does not work here). We
attempt to take $\gamma=2\ve,$ i.e. glue $F^+$ directly to $L^-$
without further transition regimes. From~\eqref{hh27.6},
$\sgn(\tau_+(u))>0,$ $\forall u\ge2\ve.$ Therefore, we have obtained
the following complete lower candidate:
\begin{figure}[ht]
  \centering{\includegraphics[width=11cm]{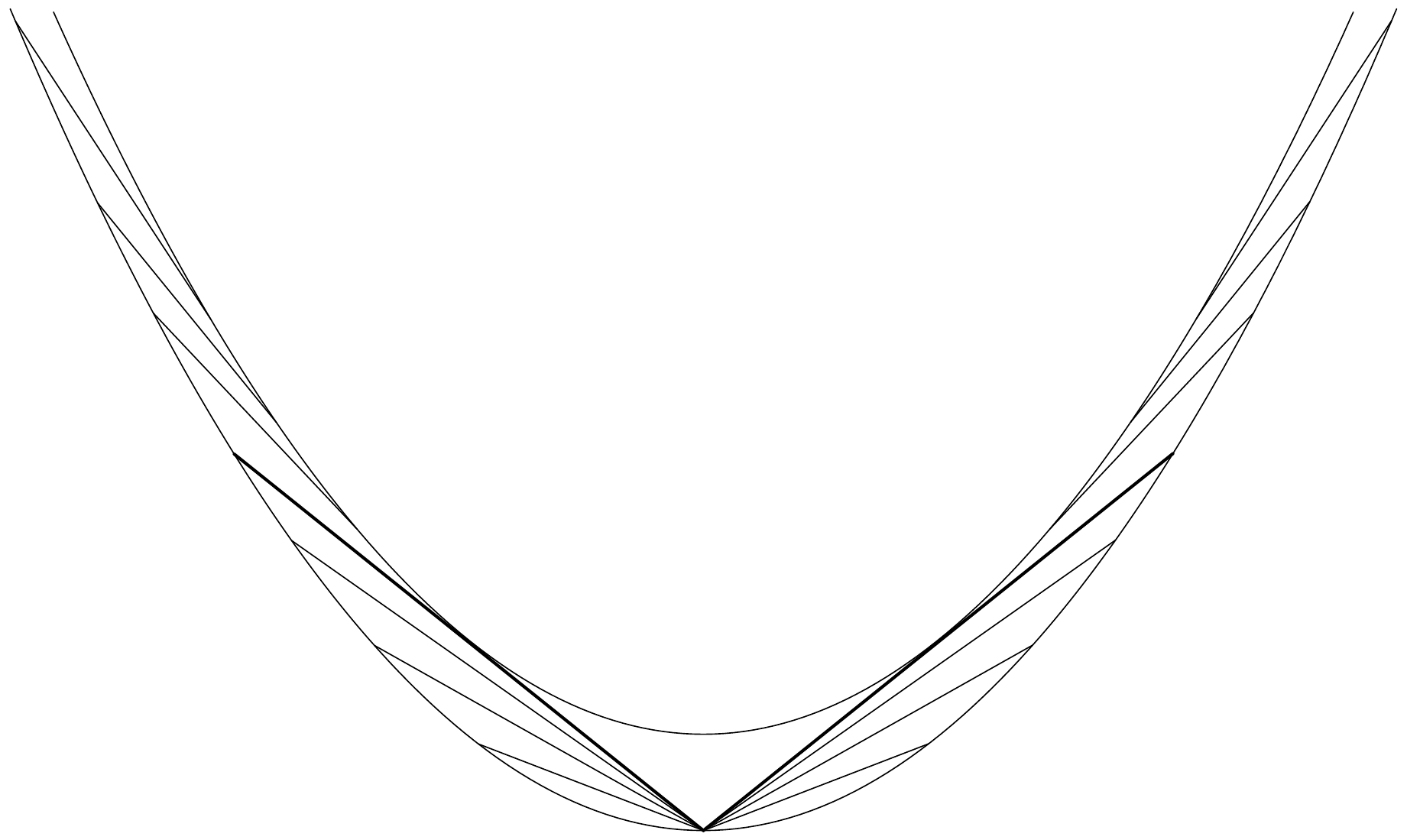}}
\caption{The Bellman foliation for the candidate $R.$}
\label{fig66}
\end{figure}
\eq[h27]{
R(x)=
\begin{cases}
F^-(x;-\infty,-2\ve),&x\in\Omega_{F^-}(-\infty,-2\ve),\\
L^+(x;-\ve),&x\in\Omega_L(-\ve),\\
T_0(x),&x\in\Omega_T(0),\\
L^-(x;\ve)&x\in\Omega_L(\ve),\\
F^+(x;2\ve,\infty),&x\in\Omega_{F^+}(2\ve,\infty).
\end{cases}
}
Written explicitly, the function $R$ is given by
\eq[h28]{
R_{\ve,p}(x)=
\begin{cases}
\ds m_+(u;2\ve)(|x_1|-u)+u^p,
&x\in\Omega_{F^-}(-\infty,-2\ve)\cup\Omega_{F^+}(2\ve,\infty),\\
x_2^{p-1}|x_1|^{2-p},&x\in\Omega_L(-\ve)\cup\Omega_L(\ve),
\rule[10pt]{0pt}{7pt}\\
\rule[10pt]{0pt}{7pt}
(2\ve)^{p-2}x_2,&x\in\Omega_T(0),
\end{cases}
}
where
$$
u=|x_1|+\ve-\sqrt{\ve^2-x_2+x_1^2}.
$$
The corresponding foliation is shown on Figure~\ref{fig66}.
\section{Bellman Induction}
\label{induction}
In this section, we first establish the local concavity/convexity properties
of the global Bellman candidates in the whole domain $\Omega_\ve$, i.e. show
that each candidate $G$ satisfies either
\eq[i1]{
G(\alpha_-x^-+\alpha_+x^+)\ge\alpha_-G(x^-)+\alpha_+G(x^+)
}
or
\eq[i2]{
G(\alpha_-x^-+\alpha_+x^+)\le\alpha_-G(x^-)+\alpha_+G(x^+)
}
for all non-negative numbers $\alpha_\pm$ such that $\alpha_-+\alpha_+=1$
and all $x^-,x^+\in\Oe$ such that the entire line segment $[x^-,x^+]$
is inside $\Oe.$ With this in hand, we then use induction
on scales to show that each candidate appropriately majorates (minorates)
the Bellman function for which it was constructed.

Before verifying~\eqref{i1} or~\eqref{i2} for our Bellman candidates,
we need to make two observations. First, note that we can
consider~\eqref{i1} or~\eqref{i2} as the statement that the derivative
of $G$ along every direction is decreasing (respectively, increasing)
in that direction; thus, each property can be checked locally. We know
from Section~\ref{global} that each global candidate $G$ 
has the Hessian of the appropriate sign --- and so
the required monotonicity of the derivatives --- in every canonical
subdomain. Therefore, the only places where~\eqref{i1} or~\eqref{i2}
needs to be verified are the points where the segment $[x^-,x^+]$
intersects the boundaries between subdomains. For each such boundary,
that verification will take the form of measuring the jump in the
derivative along any direction transversal to the boundary. While
in general our global candidates $G$ are not guaranteed to be smooth
at such points, we do have one-sided transversal derivatives everywhere
and can check the sign of that jump. If that jump turns out to be
$0$ at all points of the boundary, that means that the two solutions
are, in fact, glued $C^1$-smoothly. In this case, we will automatically
obtain a global candidate satisfying either~\eqref{i1} or~\eqref{i2}
in the union of the two subdomains.

The second important observation is the following: whenever we glue
two subdomains, it is always along a line that is an extremal trajectory
for both local foliations. For this reason, while technically we need to
check the sign of the jump in the derivative at all points of the shared
boundary and in all directions transversal to it, it suffices to do so
at a single boundary point, in a single transversal direction. Let us
explain: recall that for any solution $G$ of the \ma equation, the
gradient $(G_{x_1},G_{x_2})=(t_1,t_2)$ is constant along any extremal
line; hence, the jump in the vector $(t_1,t_2)$ is constant along the
boundary and we can measure it at any point. On the other hand, consider
new coordinates $(y_1,y_2),$ where $y_1$ is directed along the shared
boundary and $y_2,$ in a fixed transversal direction. Then $G_{y_1}$
is continuous (so no jump in that component), because both solutions are
linear functions along $y_1$ and they were glued continuously. Thus, we
only need to measure the jump in $G_{y_2}.$

We are now in a position to prove the following

\begin{lemma}
\label{conc}
~
\ben
\item
The function $M_{\ve,p}$ given by~\eqref{ff6} satisfies~\eqref{i1}
for $1\le p\le 2$ and~\eqref{i2} for $p\ge2.$
\item
The function $N_{\ve,p}$ given by~\eqref{m3} satisfies~\eqref{i2}
for $1\le p\le 2$ and~\eqref{i1} for $p\ge2.$
\item
The function $P_{\ve,p}$ given by~\eqref{h12} satisfies~\eqref{i1}
for $0<p<1.$
\item
The function $R_{\ve,p}$ given by~\eqref{h28} satisfies~\eqref{i2}
for $0<p\le1.$
\een
\end{lemma}
\begin{proof}
According to the preceding discussion, we need to check all boundaries
between subdomains for each candidate. Since such boundaries are never
parallel to the $x_2$-axis, we can choose $x_2$ as our transversal
direction in all cases. Therefore, to prove the lemma, we will use
the following procedure: for each specific candidate $G$ and each
boundary between two of its subdomains, pick one point $x=(x_1,x_2)$
on the boundary and verify that $t_2(x^+)\le t_2(x^-)$ for~\eqref{i1}
and $t_2(x^+)\ge t_2(x^-)$ for~\eqref{i2}, where $t_2(x^\pm)=
G_{x_2}(x_1,x_2\pm0).$

In all cases, by symmetry it is sufficient to check only those
subdomain boundaries that are in $\Oe^+.$
\ben
\item
\label{pp1}
In the case of $M_{\ve,p},$ we check the jump in $t_2$ along the line
$x_2=\ve^2$ separating $\Omega_L(0)$ and $\Omega_{F^+}(\ve,\infty).$
From~\eqref{ff008} and~\eqref{ff009}, we have, for any point $x$ on
this line,
$$
t_2(x^+)=\frac12m'(\ve)=\frac1{2\ve}(p\ve^{p-1}-m(\ve)),
$$
where $m(\ve)=0$ from \eqref{h4.1} with $u=u_1=\ve.$ Thus, we have
$t_2(x^+)=\frac p2\ve^{p-2}.$

On the other hand, from \eqref{ff6}  $t_2(x^-)=\frac p2x_2^{p/2-1}
=\frac p2(\ve^2)^{p/2-1}=\frac p2\ve^{p-2}.$ Therefore, the derivative
jump is zero for any $p$ and the first statement of the lemma is proved.
\item
For $N_{\ve,p},$ we check the jump in $t_2$ along the line $x_2=
2\ve x_1$ separating $\Omega_T(0)$ and $\Omega_{F^-}(0,\infty).$
From~\eqref{ff008} and~\eqref{ff010}, we have, for any point $x$
on this line,
$$
t_2(x^-)=\frac12m'(0)=\frac1{2\ve}m(0),
$$
where
$m(0)=p\ve^{p-1}\int_0^\infty t^{p-1}e^{-t}\,dt=p\ve^{p-1}\Gamma(p)$
from \eqref{h4.2} with $u=0$ and $u_2=\infty.$ Thus, we have $t_2(x^-)
=\frac p2\ve^{p-2}\Gamma(p).$

On the other hand, from \eqref{m3} $t_2(x^+)=\frac p2\ve^{p-2}\Gamma(p).$
Again, the derivative jump is zero and the second statement is proved.
\item
For $P_{\ve,p},$ we need to check three boundary lines: $x_2=\ve^2$
between $\Omega_L(0)$ and $\Omega_{F^+}(\ve,\xi);$ $x_2=2(\xi+\ve)x_1
-2\xi\ve-\xi^2$ between $\Omega_{F^+}(\ve,\xi)$ and $\Omega_T(\xi);$
and $x_2=2(\xi-\ve)x_1+2\xi\ve-\xi^2$ between $\Omega_T(\xi)$ and
$\Omega_{F^-}(\xi,\infty).$ The first verification is the same as
in part~\eqref{pp1} above. The second and third are automatic: the value of $\xi=\mu\ve$ in~\eqref{h12} was chosen according to Corollary~\ref{mu}, which ensured that
condition~\eqref{h7.1} is satisfied. That condition, in turn, was a criterion for having zero jump in $t_2$ across each bounding tangent of $\Omega_T(\xi).$
\item
For $R_{\ve,p},$ we have two segments of the same line to check: for
$0\le x_1\le \ve,$ the line $x_2=2\ve x_1$ separates $\Omega_T(0)$
and $\Omega_L(\ve);$ for $\ve\le x_1\le 2\ve,$ the same line separates
$\Omega_L(\ve)$ and $\Omega_{F^+}(2\ve,\infty).$

In the first case, $t_2(x^+)=(2\ve)^{p-2}$ and $t_2(x^-)=
(p-1)x_2^{p-2}x_1^{2-p}$ from \eqref{h28}. When $x_2=2\ve x_1,$ we
have $t_2(x^-)=(p-1)(2\ve)^{p-2}$ and so $t_2(x^+)-t_2(x^-)=
(2\ve)^{p-2}(2-p)>0,$ which is consistent with $R_{\ve,p}$ being a
lower Bellman candidate.

In the second case, $t_2(x^+)$ is given, similarly to part~\eqref{pp1},
by $t_2(x^+)=\frac12m'(2\ve)=\frac1{2\ve}(p(2\ve)^{p-1}-m(2\ve)),$
where $m(2\ve)=(2\ve)^{p-1}$ from~\eqref{h4.1} with $u=u_1=2\ve.$
Thus, $t_2(x^+)=(p-1)(2\ve)^{p-2}.$ On the other hand, $t_2(x^-)
=(p-1)(2\ve)^{p-2},$ as before, and again we have $C^1$ smoothness of
$R_{\ve,p}.$
\een
\end{proof}

We will make use of the following geometric result, whose proof can be
found in~\cite{sv}.
\begin{lemma}
\label{l4c}
Fix $\ve>0.$ Take any $\delta>\ve.$ Then for every interval $I$ and
every $\varphi\in\BMO_{\ve}(I),$ there exists a splitting $I=I_-\cup I_+$
such that the whole straight-line segment with the endpoints
$x^{\pm}=\left(\av{\varphi}{I_{\pm}},\av{\varphi^2}{I_{\pm}}\right)$ is
inside $\Omega_\delta.$ Moreover\textup, the splitting parameter
$\alpha_+=|I_+|/|I|$ can be chosen uniformly \textup(with respect to
$\varphi$ and $I$\textup) separated from $0$ and $1.$
\end{lemma}

We need another simple lemma that says that cutting a $\BMO$ function
off at a given height does not increase its norm, which is implicitly
contained in~\cite{sv} as well.

\begin{lemma}
\label{norm}
Fix $\varphi\in\BMO(\rn)$ and $c,d\in\mathbb{R}$ such that $c<d.$ Let
$\varphi_{c,d}$ be the cut-off of $\varphi$ at heights $c$ and $d:$
\eq[cutoff]{
\varphi_{c,d}(s)=
\begin{cases}
\ c,&if~\varphi(s)\le c;\\
\varphi(s),&if~c<\varphi(s)<d;\\
\ d,&if~\varphi(s)\ge d.
\end{cases}
}
Then
$$
\av{\varphi_{c,d}^2}J-\av{\varphi_{c,d}}J^2\le\av{\varphi^2}J-
\av\varphi J^2,\quad\forall\text{~cube~}J,
$$
and\textup, consequently\textup,
$$
\|\varphi_{c,d}\|_{\BMO}\le\|\varphi\|_{\BMO}.
$$
\end{lemma}

\begin{proof}
First, let us note that it is sufficient to prove this lemma for a
one-sided cut, for example, for $c=-\infty.$ We then get the full
statement by applying this argument twice. Indeed, if we denote by
$C_d\varphi$ the cut-off of $\varphi$ from above at height $d,$ i.e.
$C_d\varphi=\varphi_{-\infty,d},$ then $\varphi_{c,d}=-C_{-c}(-C_d\varphi$).

Take a cube $J$ and let $J_1=\{s\in J\colon \varphi(s)\le d\}$ and
$J_2=\{s\in J\colon \varphi(s)>d\}.$ If either $J_1=\emptyset$ or
$J_2=\emptyset,$ the statement is trivial. Thus, we may assume that
$J_k\ne\emptyset.$ Let $\beta_k=|J_k|/|J|, k=1,2.$
We have the following identity:
\begin{align*}
\bigl[\av{\varphi^2}J&-\av\varphi J^2\bigr] -\bigl[\av{(C_d\varphi)^2}J
-\av{C_d\varphi}J^2\bigr]\\
=&\beta_2\bigl[\av{\varphi^2}{J_2}-\av\varphi {J_2}^2\bigr]+
\beta_1\beta_2\bigl[\av\varphi{J_2}-d\bigr]\bigl[\av\varphi{J_2}+d-
2\av\varphi{J_1}\bigr],
\end{align*}
which proves the lemma, because $\av\varphi{J_1}\!\!\le d\le \av\varphi{J_2}.$
\end{proof}

The following is the main result of this section. Its statement is
similar to --- if much more general than --- that of  Lemma~2c
from~\cite{sv}. The proof, by induction on pseudo-dyadic scales, is
somewhat streamlined compared to that in \cite{sv},  although its
main ingredients are the same.

\begin{lemma}
\label{l2c}
~
\noindent Fix $\ve>0$ and let $B$ and $b$ be two functions defined
and continuous on $\Omega_\delta$ for some $\delta>\ve.$ Assume that
$B$ has property~\eqref{i1} and $b$ has property~\eqref{i2} on
$\Omega_\delta.$ Let $W(t)=B(t,t^2),$ $w(t)=b(t,t^2).$ If either
$W$ or $w$ is unbounded at $-\infty$ or $+\infty,$ assume it is
monotone for $t$ sufficiently close to $-\infty$ or $+\infty,$
respectively.

Fix a point $x\in\Oe$ and an interval $Q$ and take any function
$\varphi\in BMO_\ve(Q)$ such that $(\av{\varphi}Q,\av{\varphi^2}Q)=x.$
Then
$$
B(x)\ge\av{W(\varphi)}Q,
$$
$$
b(x)\le\av{w(\varphi)}Q,
$$
including the possible infinite values on either side of each inequality.
\end{lemma}

\begin{proof}
We will only prove the statement of the lemma concerning $B,$ as the part
concerning $b$ is virtually identical. We first establish the result
for those $\varphi\in\BMO_\ve(Q)$ that are bounded and then approximate
arbitrary \BMO\ functions by appropriately chosen cut-offs.

Take $\varphi\in\BMO_\ve(Q)\cap L^\infty(Q).$
Observe that $\varphi\in \BMO_{\ve}(I)$ for any subinterval $I$ of $Q.$
We now build a binary tree $D(Q)$ of subintervals of $Q,$ where every
interval $I\in D(Q)$ is split into two subintervals $I_\pm\in D(Q)$
according to the rule from Lemma~\ref{l4c}. The set of intervals of
the $n$-th generation will be denoted by $D_n(Q),$ so $D_0(Q)=\{Q\},$
$D_1(Q)=\{Q_\pm\},$ etc. For every interval $I\in D(Q),$ let
$x^I\in\Omega_\ve$ be the corresponding Bellman point, $x^I=
\left(\av\varphi I,\av{\varphi^2}I\right).$ Let $x^{(n)}$ denote the
step function from $Q$ into $\Omega_\ve,$ defined by the rule
$x^{(n)}(t)=x^I$ if $t\in I,$ $I\in D_n(Q)$. Since Lemma~\ref{l4c}
provides for the value of $\alpha_+$ uniformly separated from $0$ and $1$
on every step, we have
$$
\max_{I\in D_n(Q)}\left\{|I|\right\}\to0\quad\text{as}\quad n\to\infty.
$$
By the Lebesgue differentiation theorem, we have
$x^{(n)}(t)\to(\varphi(t),\varphi^2(t))$ almost everywhere. Since
$\varphi$ is assumed bounded, $\{x^{(n)}\}$ is a sequence of bounded
functions.

For each of the splits prescribed by Lemma~\ref{l4c}, the line segment
connecting $x^{I_-},$ $x^I,$ and $x^{I_+}$ lies in $\Omega_\delta.$
Using the property \eqref{i1} of $B$ repeatedly, we get, for any $n\ge 1,$
\begin{align*}
|Q|B(x^Q)&\ge|Q_+|B(x^{Q_+})+|Q_-|B(x^{Q_-})
\\
&\ge\sum_{I\in D_n(Q)}\!\!\!|I|B(x^I)=\int_Q \!\!B(x^{(n)}(t))\,dt.
\end{align*}
Since $B$ is continuous on $\Omega_\delta$ (and thus on $\Oe$), the
dominated convergence theorem applies and taking the limit as
$n\to\infty$ proves the lemma for bounded $\varphi.$

Take now an arbitrary $\varphi\in\BMO_\ve(Q).$ For $c, d\in\mathbb{R}$
such that $c<d,$ let $\varphi_{c,d}$ be defined by~\eqref{cutoff}. We
have $\varphi_{c,d}\in L^\infty(Q)$ and, by Lemma~\ref{norm},
$\|\varphi_{c,d}\|_\BMO\le\ve.$ Therefore,
$$
B(\av{\varphi_{c,d}}Q,\av{\varphi^2_{c,d}}Q)\ge\av{W(\varphi_{c,d})}Q.
$$
We now take the limit in this inequality as $c\to-\infty.$ Since $B$
is continuous, the limit of the left-hand side is
$B(\av{\varphi_{-\infty,d}}Q,\av{\varphi^2_{-\infty,d}}Q),$ where
$\varphi_{-\infty,d}\df d\chi_{\{\varphi\ge d\}}+
\varphi\chi_{\{\varphi< d\}}.$ On the other hand, by the monotone
convergence theorem, the limit of the right-hand side is
$\av{W(\varphi_{-\infty,d})}Q.$ The same argument works if we let
$d\to\infty,$ which completes the proof.
\end{proof}

As an immediate corollary, we obtain the following
\begin{theorem}
\label{imp1}
For any $x\in\Oe,$ we have
$$
\begin{array}{lll}
For~p\ge2:\quad
&N_{\ve,p}(x)\ge\bel{B}_{\ve,p}(x),
&\quad M_{\ve,p}(x)\le\bel{b}_{\ve,p}(x),
\\
For~1\le p< 2:\quad
&M_{\ve,p}(x)\ge\bel{B}_{\ve,p}(x),
&\quad N_{\ve,p}(x)\le\bel{b}_{\ve,p}(x),
\\
For~0<p<1:\quad
&P_{\ve,p}(x)\ge\bel{B}_{\ve,p}(x),
&\quad R_{\ve,p}(x)\le\bel{b}_{\ve,p}(x).
\end{array}
$$
\end{theorem}
\begin{proof}
Let $B_\ve$ stand for any of the upper candidates in the statement of
the theorem and $b_\ve$ for any of the lower candidates. Since
each $B_\ve$ and each $b_\ve$ is continuous in $\ve,$ it is sufficient
to prove that
$$
B_\delta(x)\ge\bel{B}_{\ve,p}(x),\quad
b_\delta(x)\le\bel{b}_{\ve,p}(x),\quad\forall x\in\Oe,
$$
for all $\delta>\ve,$ and then let $\delta\to\ve.$

Observe that each $B_\delta$ and each $b_\delta$ satisfies the conditions
of Lemma~\ref{l2c} and $B_\delta(t,t^2)=b_\delta(t,t^2)=|t|^p.$ Therefore,
$$
B_\delta(x)\ge\av{|\varphi|^p}Q\,,\qquad b_\delta(x)\le\av{|\varphi|^p}Q\,,
$$
for all $\varphi\in\BMO_\ve(Q)$ with $(\av{\varphi}Q,\av{\varphi^2}Q)=x.$
Now, take the supremum over all such $\varphi$ in the first inequality
and infimum in the second.
\end{proof}
\section{Optimizers and converse inequalities}
\label{converse}
In this section we construct, for each Bellman candidate $G$ built in
Section~\ref{global} and each point $x\in\Oe,$ a test function
$\varphi_x$ on $(0,1)$ with the following three properties:
\eq[co1]{
\begin{array}{l}
(a)\quad\av{\varphi_x}{(0,1)}\!\!=x_1,\quad
\av{\varphi_x^2}{(0,1)}\!\!=x_2;
\\
(b)\quad \varphi_x\in\BMO_\ve((0,1));\rule{0pt}{15pt}
\\
(c)\quad \av{f(\varphi_x)}{(0,1)}\!\!=G(x).\rule{0pt}{18pt}
\end{array}
}
We will call each such function $\varphi_x$ {\it an optimizer} for
$G(x).$ We will often need optimizers for points on the top boundary
of $\Oe,$ the parabola $x_2=x_1^2+\ve^2.$ It is convenient to
parametrize these by the horizontal coordinate: let us denote
$\varphi_{(x_1,x_1^2+\ve^2)}$ by $\psi_{x_1}.$ For the rest of
this section, we will simply write $\BMO_\ve$ for $\BMO_\ve((0,1)).$
If we need $\BMO$ over another interval, we will write the interval
explicitly.

If the candidate $G$ corresponds to the Bellman function $\bel{G},$ the existence of $\varphi$ satisfying~\eqref{co1}
would immediately imply that
\eq[co2]{
\begin{array}{ll}
G(x)\le\bel{G}(x),\quad\forall x\in\Oe,&\hbox{if $\bel{G}$ is an upper \Bf},
\rule{0pt}{15pt}
\\
G(x)\ge\bel{G}(x),\quad\forall x\in\Oe,&\hbox{if $\bel{G}$ is a lower \Bf}.
\rule{0pt}{15pt}
\end{array}
}
\smallskip

In general, whether an extremal function exists for a particular
inequality is a deep question, which sometimes is more difficult
to answer than to prove the inequality itself. However, if the \Bf\
for the inequality is known, one has a definitive answer as to the
existence and the nature of optimizers. If one exists, it is found
using the geometry of the Bellman foliation, as done in this section.
If one does not exist (which is the case, for example, for all known
\Bfs\ for the maximal operator), optimizing sequences are built along
the trajectories of the foliation.

Since our global Bellman candidates are built out of four canonical
local blocks, $F^+,\;F^-,\;L,$ and/or $T,$ each with its own geometry,
it is natural to attempt to construct a corresponding set of canonical
optimizers, one for each local candidate. If one has such a set, one
can demonstrate that (the appropriate line of)~\eqref{co2} holds for
the global candidate $G$ simply by showing it for each canonical
sub-domain of $\Oe.$

Thus, we fix a canonical block $D$ and attempt to build an optimizer
$\varphi_x$ for each point of the domain $\Omega_D.$ Expectedly,
$\varphi_x$ will, generally, depend not only on $x,$ but also on
the placement of $\Omega_D$ within $\Oe.$ This is so because our
local candidates, for the most part, themselves explicitly depend on
that placement. Somewhat more subtly, the optimizers turn out to depend,
in some cases, on exactly how the local candidates are glued together
to produce a global one. In Section~\ref{local}, we built blocks of two
kinds. Some were completely determined by the parameters of their
domain, namely $F^+(x;-\infty,u_2),$ $F^-(x;u_1,\infty),$ and all
blocks of the $L$ type; let us call these blocks {\it complete}. Others,
namely $F^+(x;u_1,u_2)$ for $u_1>-\infty,$ $F^-(x;u_1,u_2)$  for
$u_2<\infty,$  and $T$, had undetermined constants that were only
found in Section~\ref{global}, where the neighbors of these blocks
in the global context were examined; let us call such blocks {\it
incomplete}. Furthermore, different incomplete blocks may or may not
require the knowledge of both neighbors. Thus $F^+$ is {\it
left-incomplete}, unless $u_1=-\infty,$ since it requires the
knowledge if its left neighbor (which, from among our limited supply,
can only be an $L$ block); $F^+$ is {\it right-incomplete}, unless
$u_2=\infty,$ as it requires the neighbor on the right (again, an $L$
block); and $T$ is both left- and right-incomplete, as it requires both
neighbors (these can be various combinations of $L$ and $F$ blocks).

This division has exact parallels in this section: to determine an
optimizer for a complete block, we only need to know its domain. For
example, it is meaningful to say that a given function is an optimizer for
$L^+(x;a).$ However, to determine an optimizer for an incomplete
block, we need to know its neighbors, and their optimizers, on
the left and/or on the right, all the way to a complete block. Thus,
we cannot say that a given function $\varphi_x$ is an optimizer for
$F^+(x;u_1,u_2),$ because $F^+$ is incomplete. We could,
however, say that a function is an optimizer for the sequence
$L(x;u_1\!-\!\ve)\to F^+(x;u_1,u_2),$ because such a sequence
determines the candidate completely.

In such situations, we will talk about {\it optimizers for a block}
(alternatively, {\it local candidate\textup) in the context of a
sequence}. Rather than defining this notion in general, let us list
the specific sequences we will encounter, including those consisting
of a single block: $L,$ $F^+(x;-\infty,u_2),$ $F^-(x;u_1,\infty),$
$L\to F^+,$ $F^-\!\!\leftarrow L,$ $L\to T\leftarrow L,$ $F^+\!\!\to
T\leftarrow F^-,$ and $L\to F^+\!\!\to T\leftarrow F^-(x;u_1,\infty).$
In each case, all arrows point to the block for which an optimizer is
being sought. Expectedly, each such sequence starts with a block that
is complete on the left and ends with one complete on the right.

The main empirical consideration in building optimizers for Bellman
foliations is that for each point $x$ the whole construction of the
optimizer $\varphi_x$ should be taking place along the extremal
trajectory passing through $x.$ This means that $\varphi_x$ should
be such that when the interval $I=(0,1)$ is split into two subintervals, 
$I=I_-\cup I_+,$  the Bellman points $x^{I_\pm}=(\av{\varphi_x}{I_\pm}\!,
\av{\varphi_x^2}{I_\pm}\!)$ are also on the trajectory. Indeed,
for the true \Bf, each inequality in the Bellman induction of
Section~\ref{induction} is, in fact, an equality; thus, each split
must be happening along the trajectory where the candidate is linear.
Conversely, if $m,n$ are two points on such a trajectory and we know
the optimizers $\varphi_{m}$ and $\varphi_{n},$ then the optimizer for
any point $x$ on the trajectory that is between $m$ and $n$ can be
obtained simply by concatenating the two known optimizers:
$$
\varphi_x(t)=
\begin{cases}
\varphi_m\left(\frac t\gamma\right),&t\in(0,\gamma),
\\
\varphi_n\left(\frac{t-\gamma}{1-\gamma}\right),&t\in(\gamma,1),
\rule{0pt}{18pt}
\end{cases}
$$
where $\gamma=(x_1-n_1)/(m_1-n_1).$
\begin{remark}
\label{care}
In general, concatenating two functions from $\BMO_\ve$ is not
guaranteed to produce another $\BMO_\ve$ function. One or both of
the functions being concatenated, $\varphi_n$ and $\varphi_m$ in the
formula above,  may have to be rearranged to ensure the smallest
possible $\BMO$ norm of the resulting function.
\end{remark}
\begin{remark}
As we will see, the optimizers built in this section {\it do not
depend} on the choice of the boundary function $f, $ but only on
the geometry of the canonical subdomain and, in some cases, on how
such subdomains are glued together.
\end{remark}
\subsection{Optimizers for $L_0(x)$} The foliation of the domain
$\Omega_L(0)$ that corresponds to the Bellman candidate $L_0$ given
by~\eqref{g1} consists of horizontal lines. For each point
$x\in\Omega_L(0),$ the extremal trajectory through $x$ intersects
the boundary of $\Oe$ in two points, $(-\sqrt{x_2},x_2)$ and
$(\sqrt{x_2},x_2).$ Since the only test functions available on the
boundary of $\Oe$ are constants, we already know the optimizers for
these points: $\varphi_{(\pm\sqrt{x_2},x_2)}(t)=\pm\sqrt{x_2}.$
Therefore, to construct the optimizer for the point $x,$ we
concatenate the two boundary values in the appropriate proportion:
\eq[ol01]{
\varphi_x(t)=
\begin{cases}
-\sqrt{x_2},&\ds t\in(0,\alpha),\\
\ \sqrt{x_2},&\ds t\in(\alpha,1),
\end{cases}
}
where
\eq[ol02]{
\alpha=\frac12\left(1-\frac{x_1}{\sqrt{x_2}}\right).
}
We now verify~\eqref{co1} for $\varphi_x.$
\begin{lemma}
\label{oL0}
The function $\varphi_x$ given by~\eqref{ol01} and~\eqref{ol02} is an
optimizer for $L_0(x).$
\end{lemma}
\begin{proof}
For part~(a) of~\eqref{co1}, we have $\av{\varphi_x}{(0,1)}\!\!=
-\alpha\sqrt{x_2}+(1-\alpha)\sqrt{x_2}=x_1$ and $\av{\varphi_x^2}{(0,1)}
\!\!=\alpha x_2+(1-\alpha)x_2=x_2.$

To show (b), we need to show that for any subinterval $I$ of $(0,1)$ the
Bellman point $x^I,$ corresponding to $I$ and $\varphi_x,$ is in $\Oe.$
Observe that $x^I$ is a convex combination of two boundary points,
$(\pm\sqrt{x_2},x_2).$ Thus, it lies on the line segment connecting
these points and, furthermore, belongs to the convex subset  $\Omega_L(0)$ of
$\Oe.$

For (c), we trivially have
$\ds\av{f(\varphi_x)}{(0,1)}\!\!=f(\sqrt{x_2})=L_0(x).$
\end{proof}
\subsection{Optimizers for $L^\pm(x;a),$ $a\ne0$}
Recall that for each choice of $a$ the candidates $L^\pm(x;a),$ given by \eqref{g21} and \eqref{g20.0}, are
built on the domain $\Omega_L(a)$ consisting of the (closed) portion
of $\Oe$ lying under a two-sided tangent $x_2=2ax_1+\ve^2-a^2.$ The Bellman foliation for $L^+$ is the collection of
straight lines connecting the corner point $(u_+,u_+^2)$ to points
$(u,u^2)$ with $u_-<u<u_+,$ while the foliation for $L^-$ consists of
lines connecting $(u_-,u_-^2)$ to $(u,u^2).$ For each $x\in\Omega_L(a),$
except the two corners, there is a unique such $u,$ given
by~\eqref{g20.0}:
\eq[ol1]{
u=\frac{x_2-vx_1}{x_1-v},
}
where $v$ stands for either $u_+$ or $u_-,$ as appropriate.

Again, we already know the optimizers for the points $(v,v^2)$ and
$(u,u^2):$ $\varphi_{(v,v^2)}(t)=v$ and $\varphi_{(u,u^2)}(t)=u,$
respectively. Therefore, to construct the optimizer for the point
$x$ lying on the line connecting $v$ and $u,$ we concatenate these
two optimizers:
\eq[ol2]{
\varphi_x(t)=
\begin{cases}
u,&t\in(0,\alpha);\\
v,&t\in(\alpha,1),
\end{cases}
}
where
\eq[ol3]{
\alpha=\frac{x_1-v}{u-v}=\frac{(x_1-v)^2}{x_2-2vx_1+v^2}.
}
We now verify \eqref{co1} for $\varphi_x.$
\begin{lemma}
The function $\varphi_x$ given by~\eqref{ol1}\textup,
\eqref{ol2}\textup, and~\eqref{ol3} is an optimizer for $L^-(x;a),$
if $v=u_-,$ and for $L^+(x;a),$ if $v=u_+.$
\end{lemma}
\begin{proof}
For~(a), we have $\av{\varphi_x}{(0,1)}\!\!=\alpha u+(1-\alpha)v=x_1$
and $\av{\varphi_x^2}{(0,1)}\!\!=\alpha u^2+(1-\alpha)v^2=x_2.$

For~(b), we proceed as before. Let $x^I$ be the Bellman point
corresponding to $\varphi_x$ and a subinterval $I$ of $(0,1).$
This point is a convex combination of the points $(v,v^2)$ and
$(u,u^2),$ which lie in the convex subset $\Omega_L(a)$ of $\Oe.$
Therefore, $x^I$ belongs to $\Oe.$

Finally, we check the optimality relation~(c) of~\eqref{co1}:
$$
\av{f(\varphi)}{(0,1)}=\alpha f(u)+(1-\alpha)f(v)=
\frac{f(v)-f(u)}{v-u}\,(x_1-v)+f(v)=L(x;a),
$$
according to~\eqref{g21}. Here $L$ stands for either $L^+$ or $L^-,$
as appropriate.
\end{proof}
\subsection{Optimizers for $L(x;u_1\!-\ve)\to F^+(x;u_1,u_2)$}
Recall formula \eqref{gg4.0} for $F^+$ in this context:
$$
F^+(x;u_1,u_2)=\frac1\ve e^{-u/\ve}
\left[\frac{f(u_1)-f(u_1-2\ve)}2 e^{u_1/\ve}+
\int_{u_1}^u\!\!f'(s)e^{s/\ve}\,ds\right]\, (x_1-u)+f(u),
$$
with $u$ given by \eqref{ff5}:
$$
u=u_+\!=x_1+\ve-\sqrt{\ve^2-x_2+x_1^2}.
$$

The extremal trajectories for $F^+$ are one-sided tangents to the
upper boundary. According to the discussion in the introduction
to this section, to be able to construct an optimizer for any
$x\in\Omega_{F^+}(u_1,u_2),$ we first construct one for each point
of the upper boundary of this sub-domain, $\{(a,a^2+\ve^2),$
$a\in(a_1,a_2]\},$ $a_i=u_i-\ve.$ 
Then, we concatenate the optimizer on the upper boundary
with the constant optimizer on the bottom boundary via
the extremal tangent through $x.$

Fix an $a\in(a_1,a_2]$ and consider the tangent $x_2=2ax_1+\ve^2-a^2$
intersecting the lower boundary at the point $(u,u^2).$ Since all our
extremal tangents are one-sided, we do not seem to have another
optimizer with which to concatenate the constant function $u.$
However, we circumvent this difficulty with the following approximating
procedure.  Fix a small number $\Delta$ and let $\alpha=
\frac{\ve-\Delta}{\ve+\Delta},$ $\beta=\frac\ve{\ve+\Delta}.$
Now consider the point $(a-\Delta,a^2+\ve^2-2a\Delta),$ also on
the tangent. Assume for a moment that we know the optimizer $\rho$
for this point and set
\eq[co3]{
\psi_a(t)\approx
\begin{cases}
\rho\left(\frac t\beta\right),&t\in(0,\beta),\\
u,&t\in(\beta,1).
\end{cases}
}
To get the optimizer $\rho,$ we draw a tangent through the point
$(a-\Delta,a^2+\ve^2-2a\Delta),$ which intersects the upper boundary
at the point $(a-2\Delta,(a-2\Delta)^2+\ve^2)$ and the lower boundary,
at the point $(u-2\Delta,(u-2\Delta)^2)$ (see Figure~\ref{fig 0.166}). If we knew the optimizer
$\psi_{a-2\Delta},$ we could again concatenate the optimizers on the
two boundaries:
\eq[co4]{
\rho(t)=
\begin{cases}
\psi_{a-2\Delta}\left(\frac\ve{\ve-\Delta}t\right),&t\in(0,1-\Delta/\ve),
\\
u-2\Delta,&t\in(1-\Delta/\ve,1).
\end{cases}
}
Combining~\eqref{co3} and~\eqref{co4}, we obtain a recursive
approximation for $\psi_a:$
$$
\psi_a(t)\approx
\begin{cases}
\psi_{a-2\Delta}\left(\frac t\alpha\right),&t\in(0,\alpha),
\\
u-2\Delta,&t\in(\alpha,\beta),
\\
u,&t\in(\beta,1).
\end{cases}
$$
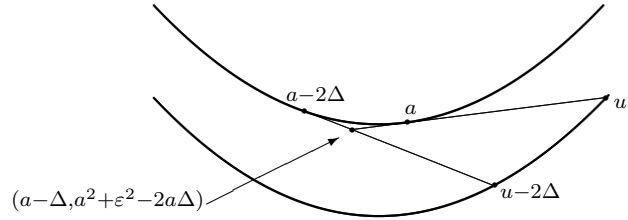
\begin{figure}[ht]
\begin{picture}(150,100)
\thicklines
\qbezier(-10,50)(75,-40)(162,51)
\qbezier(-10,85)(75,-5)(160,85)
\thinlines
\qbezier(65,38)(112,44)(160,50)
\qbezier(47,45)(83,31)(118,17)
\put(10,10){\vector(2,1){50}}
\put(86,41){\circle*{2}}
\put(65,38){\circle*{2}}
\put(47,45){\circle*{2}}
\put(119,17){\circle*{2}}
\put(161,50){\circle*{2}}
\put(84,44){$\scriptstyle a$}
\put(39,49){$\scriptstyle a-2\Delta$}
\put(163,46){$\scriptstyle u$}
\put(120,12){$\scriptstyle u-2\Delta$}
\put(-65,9){$\scriptstyle (a-\Delta,a^2+\ve^2-2a\Delta)$}
\end{picture}
\caption{The construction of $\psi_a$}
\label{fig 0.166}
\end{figure}

Repeating this procedure $k$ times, we get
\eq[co6]{
\psi_a(t)\approx\psi^{(k)}_a(t)\df
\begin{cases}
\psi_{a-2k\Delta}\left(\frac t{\alpha^k}\right),&t\in(0,\alpha^k),
\\
u-2k\Delta,&t\in(\alpha^k,\alpha^{k-1}\beta),
\\
u-2(k-1)\Delta,&t\in(\alpha^{k-1}\beta,\alpha^{k-2}\beta),
\\
\dots&\dots
\\
\dots&\dots
\\
u-2\Delta,&t\in(\alpha\beta,\beta),
\\
u,&t\in(\beta,1).
\end{cases}
}

We seem to have a major problem: we do not know $\psi_{a-2k\Delta}.$
We do, however, know $\psi_{a_1}.$ Setting $\Delta=\frac{a-a_1}{2k}=
\frac{u-u_1}{2k},$ we obtain a known function on the top line
of~\eqref{co6}. Now it is time to let $k\to\infty$ (i.e. $\Delta\to0$).
To determine $\psi_a=\lim_{k\to\infty}\psi_a^{(k)},$ we write down a
simple differential equation: take a $j,$ $1<j<k,$ and let
$t=\alpha^j\beta$ be a generic point in $(0,1).$ Then
$$
\psi_a(t)-\psi_a(\alpha t)\approx
\psi^{(k)}_a(t)-\psi^{(k)}_a(\alpha t)
\approx(u-2j\Delta)-(u-2(j+1)\Delta)=2\Delta.
$$
On the other hand,
$$
\psi_a(t)-\psi_a(\alpha t)\approx\psi'_a(t)t(1-\alpha)
\approx\psi'_a(t)t\,\frac{2\Delta}\ve.
$$
Combining the two approximate equalities and solving the differential
equation, we obtain
$$
\psi_a(t)=D+\ve\log t.
$$
Observe that
$$
\lim_{k\to\infty}\alpha^k=\lim_{\Delta\to0}
\left(1-\frac{2\Delta}{\ve+\Delta}\right)^{(a-a_1)/(2\Delta)}
\!\!\!=e^{(a_1-a)/\ve}.
$$
Altogether,~\eqref{co6} becomes
\eq[co8]{
\psi_a(t)=
\begin{cases}
\psi_{a_1}\left(e^{(a-a_1)/\ve}t\right),
&t\in\left(0,e^{(a_1-a)/\ve}\right),
\\
D+\ve\log t,&t\in\left(e^{(a_1-a)/\ve},1\right).\rule{0pt}{15pt}
\end{cases}
}

How do we determine the constant $D?$ Recall that we {\it a priori}
have $\av{\psi_{a_1}}{(0,1)}\!\!=a_1$ and must ensure
$\av{\psi_a}{(0,1)}\!\!=a.$ This gives
$$
a_1e^{(a_1-a)/\ve}+(D-\ve)\left(1-e^{(a_1-a)/\ve}\right)+
e^{(a_1-a)/\ve}(a-a_1)=a
$$
or
$$
D=a+\ve=u\,.
$$

Having constructed an optimizer for each point on the upper boundary,
we are in a position to construct one for any point
$x\in\Omega_{F^+}(u_1,u_2).$ As planned, we consider the extremal tangent
through $x,$ and concatenate the optimizer $\psi_a$ for the upper boundary
and the constant $u$ for the lower
boundary. Specifically, we have
$$
\varphi_x(t)=
\begin{cases}
\psi_a\left(\frac\ve{u-x_1}t\right),&t\in\left(0,\frac{u-x_1}\ve\right),
\\
u,&t\in\left(\frac{u-x_1}\ve,1\right).
\end{cases}
$$
Using~\eqref{co8}, we obtain the complete expression for the optimizer:
\eq[co9]{
\varphi_x(t)=
\begin{cases}
\psi_{a_1}\big(\frac t{\mu\nu}\big),&t\in\left(0,\mu\nu\right),
\\
u+\ve\log \frac t\mu,&t\in(\mu\nu,\mu),\rule{0pt}{12pt}
\\
u,&t\in\left(\mu,1\right),\rule{0pt}{12pt}
\end{cases}
}
where
\eq[co10]{
\mu=\frac{u-x_1}\ve,\quad\nu=e^{(u_1-u)/\ve},\quad\hbox{and}
\quad u=u_+\!=x_1+\ve-\sqrt{\ve^2-x_2+x_1^2}.
}
Let us now recall the expression for $\psi_{a_1},$ the optimizer for
$L(a_1,a_1^2+\ve^2;a_1).$ It is given by either~\eqref{ol01}
or~\eqref{ol2}, in each case with $\alpha=1/2:$
$$
\psi_{a_1}(t)=
\begin{cases}
a_1-\ve
\\
a_1+\ve
\end{cases}
=
\begin{cases}
u_1-2\ve,&t\in (0,1/2),
\\
u_1,&t\in (1/2,1).
\end{cases}
$$
Therefore, formula~\eqref{co9} can be rewritten as
\eq[co10.5]{
\varphi_x(t)=
\begin{cases}
u_1-2\ve,&t\in\left(0,\mu\nu/2\right),\\
u_1,&t\in\left(\mu\nu/2,\mu\nu\right),\\
u+\ve\log \frac t\mu,&t\in(\mu\nu,\mu),\\
u,&t\in\left(\mu,1\right).
\end{cases}
}

We can now prove
\begin{lemma}
\label{LF+}
The function $\varphi_x$ given by~\eqref{co10} and~\eqref{co10.5} is
an optimizer for\\ $L(x;u_1\!-\!\ve)\longrightarrow F^+(x;u_1,u_2).$
\end{lemma}
\begin{proof}
First observe that
$$
\int_{\mu\nu}^\mu\!\!\log\frac t\mu\,dt=-\mu\nu\log\nu-\mu+\mu\nu,\quad
\int_{\mu\nu}^\mu\!\!\log^2\frac t\mu\,dt
=-\mu\nu\log^2\nu+2\mu\nu\log\nu+2(\mu-\mu\nu).
$$
Now, using \eqref{co10} and the fact that $\av{\psi_{a_1}}{(0,1)}\!\!=a_1,
\av{\psi^2_{a_1}}{(0,1)}\!\!=a_1^2+\ve^2,$ we get
\begin{align*}
\av{\varphi_x}{(0,1)}&=\frac{\mu\nu}2(u_1-2\ve)+\frac{\mu\nu}2u_1
+u(1-\mu\nu)+\ve(-\mu\nu\log\nu-\mu+\mu\nu)
\\
&=\mu\nu(u_1-\ve)+u(1-\mu\nu)-\mu\nu(u_1-u)-(u-x_1)+\ve\mu\nu
\\
&=x_1
\end{align*}
and
\begin{align*}
\av{\varphi^2_x}{(0,1)}&=
\frac{\mu\nu}2(u_1-2\ve)^2+\frac{\mu\nu}2u_1^2
+u^2(1-\mu\nu)+2u\ve(-\mu\nu\log\nu-\mu+\mu\nu)
\\
&\quad+\ve^2(-\mu\nu\log^2\nu+2\mu\nu\log\nu+2\mu-2\mu\nu)
\\
&=x_2.
\end{align*}

Proving that $\varphi_x\in\BMO_\ve$ is more delicate. Let
$q(t)=u+\ve\log(t/\mu).$ Observe that $q\in\BMO_\ve((0,\infty)).$
Indeed, for $l(t)\df\log t$ and any interval $(c,d)\subset(0,\infty)$
a direct calculation yields
\eq[co10.6]{
\av{l^2}{(c,d)}-\av{l}{(c,d)}^2=
1-\frac{cd}{(d-c)^2}\log^2\left(\frac dc\right)\le1,
}
and so $l\in\BMO_1((0,\infty)),$ immediately implying the result for
$q.$ Moreover, setting $c=0$ in~\eqref{co10.6} shows that the Bellman
point corresponding to $q$ and any interval of the form $(0,\gamma)$
is always on the upper boundary of $\Oe.$ The last bit of information
we will need about $q,$ one that can be shown by a computation similar
to that at the beginning of the proof, is that
\eq[co10.7]{
\av{q}{(0,\mu\nu)}=a_1,\quad \av{q^2}{(0,\mu\nu)}=a_1^2+\ve^2.
}

Next, since $\varphi_x$ is the cut-off at height $u$ of the function
$$
\eta_x(t)\df
\begin{cases}
u_1-2\ve,&t\in\left(0,\mu\nu/2\right),\\
u_1,&t\in\left(\mu\nu/2,\mu\nu\right),\\
q(t),&t\in(\mu\nu,1),
\end{cases}
$$
it suffices to show that $\eta_x\in\BMO_\ve,$ as Lemma~\ref{norm} will
then imply the result for $\varphi_x.$ Thus, we aim to show that
$\delta_{c,d}\df\av{\eta_x^2}{(c,d)}-\av{\eta_x}{(c,d)}^2\le\ve^2$
for all $(c,d)\subset(0,1).$

We note that the only subintervals $(c,d)$ that need to be considered
in detail are those with $c\in(0,\mu\nu/2)$ and $d\in(\mu\nu,1).$
Indeed, if $0\le c<d\le\mu\nu,$ then $\delta_{c,d}\le\ve^2,$ since
$\psi_{a_1}\in\BMO_\ve,$ as has already been shown. If
$\mu\nu/2\le c<d\le 1,$ then, again, $\delta_{c,d}\le\ve^2,$ since
$\eta_x|_{(\mu\nu/2,1)}$ is the cut-off, at height $u_1,$ of the
$\BMO_\ve((\mu\nu/2,1))$ function $q$ and so Lemma~\ref{norm} again
applies.

Thus we focus on the case $0<c<\mu\nu/2<\mu\nu<d<1.$ Let
$z^-=(\av{\eta_x}{(0,c)},\av{\eta_x^2}{(0,c)}),$
$z=(\av{\eta_x}{(0,d)},\av{\eta_x^2}{(0,d)}),$ and
$z^+=(\av{\eta_x}{(c,d)},\av{\eta_x^2}{(c,d)})$ be the three Bellman
points corresponding to the intervals $(0,c),$ $(0,d),$ and $(c,d),$
respectively. Since $\eta_x|_{(0,c)}$ is constant, $z^-$ is on the
lower boundary of $\Oe.$ To locate $z,$ we note that by~\eqref{co10.7}
we have
$$
(\av{q}{(0,\mu\nu)},\av{q^2}{(0,\mu\nu)})=
(\av{\eta_x}{(0,\mu\nu)},\av{\eta_x^2}{(0,\mu\nu)}),
$$
and, therefore,
$$
(\av{q}{(0,d)},\av{q^2}{(0,d)})=(\av{\eta_x}{(0,d)},\av{\eta_x^2}{(0,d)}),
$$
for all $d\ge\mu\nu.$ As noted above, $z$ must be on the upper
boundary of $\Oe.$

Consider now the line through $z^-$ and $z.$ Since $z$ is, by
construction, to the right of the point $(a_1,a_1^2+\ve^2),$ this
line lies above the top boundary of $\Omega_L(a_1)$ and so first exits
$\Oe$ and then re-enters it at $z.$ Since $z$ is a convex combination
of $z^-$ and $z^+,$ $z^+$ is located on the same line, to the right
of $z,$ and, therefore, inside $\Oe$ (see Figure~\ref{fig 0.16}).

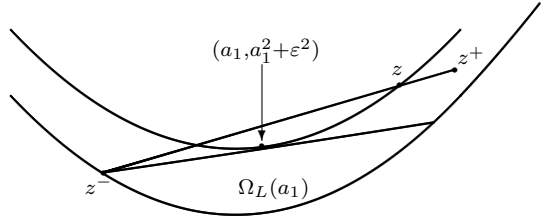
\begin{figure}[ht]
\begin{picture}(150,100)
\thicklines
\qbezier(-10,60)(75,-30)(160,60)
\qbezier(-10,85)(75,-5)(160,85)
\qbezier(25,31)(87,40)(150,50)
\qbezier(160,60)(175,76)(190,95)
\qbezier(25,31)(25,31)(158,70)
\thinlines
\put(85,72){\vector(0,-1){29}}
\put(65,75){$\scriptstyle (a_1,a_1^2+\ve^2)$}
\put(85,41){\circle*{2}}
\put(133,68){$\scriptstyle z$}
\put(17,23){$\scriptstyle z^-$}
\put(158,72){$\scriptstyle z^+$}
\put(75,23){$\scriptstyle \Omega_L(a_1)$}
\put(137,64){\circle*{2}}
\put(25,31){\circle*{2}}
\put(158,70){\circle*{2}}
\end{picture}
\caption{The mutual location of $z^-,$  $z,$ and $z^+$}
\label{fig 0.16}
\end{figure}

For part (c) of \eqref{co1}, we have
\begin{align*}
\av{f(\varphi_x)}{(0,1)}&=\frac{\mu\nu}2f(u_1-2\ve)+\frac{\mu\nu}2f(u_1)
+\int_{\mu\nu}^\mu\!\! f(u+\ve\log(t/\mu))\,dt+f(u)(1-\mu)
\\
&=\frac{\mu\nu}2f(u_1-2\ve)+\frac{\mu\nu}2f(u_1)+
\frac\mu\ve\int_{u_1}^u\!\! e^{(s-u)/\ve}f(s)\,ds+f(u)(1-\mu)
\\
&=\frac12e^{(u_1-u)/\ve}\mu\left[f(u_1-2\ve)-f(u_1)\right]
-\mu\int_{u_1}^u\!\! e^{(s-u)/\ve}f'(s)\,ds+f(u)
\\
&=F^+(x;u_1,u_2).
\end{align*}
Here we have used the formula~\eqref{co10.5} for $\varphi_x$ on the
first step, the change of variables $s=u+\ve\log(t/\mu)$ on the second,
integration by parts on the third, and expressions~\eqref{co10} for
$\mu$ and $\nu$ on the third and fourth. This completes the proof
of the lemma.
\end{proof}
\subsection{Optimizers for $F^+(x;-\infty,u_2)$}
This case follows directly from the previous one. Here the candidate
$F^+$ is given by
$$
F^+(x;-\infty,u_2)=\frac1\ve e^{-u/\ve}
\left[\int_{-\infty}^u\!\!\!f'(s)e^{s/\ve}\,ds\right]\, (x_1-u)+f(u),
$$
with $u$ given, as before, by
$$
u=u_+\!=x_1+\ve-\sqrt{\ve^2-x_2+x_1^2}.
$$
We readily obtain an optimizer for this case by setting $\nu=0$
in~\eqref{co10.5}:
\eq[co10.51]{
\varphi_x(t)=
\begin{cases}
u+\ve\log \frac t\mu,&t\in(0,\mu),\\
u,&t\in\left(\mu,1\right).
\end{cases}
}

Lemma~\ref{LF+} can be restated for this case:
\begin{lemma}
\label{F+}
The function $\varphi_x$ given by \eqref{co10.51} is an optimizer for
$F^+(x;-\infty,u_2).$
\end{lemma}

\subsection{Optimizers for sequences containing $F^-$} To construct
optimizers for $F^-$ we use the symmetry $F^-(x_1,x_2;u_1,u_2)=
F^+(-x_1,x_2;-u_2,-u_1)$. Therefore, in the optimizers given
by~\eqref{co10.5} and~\eqref{co10.51} we have to replace $\varphi$
by $-\varphi$ and $u_1$ by $-u_2$. This yields
\eq[co12.5]{
\varphi_x(t)=
\begin{cases}
u_2+2\ve,&t\in\left(0,\mu\nu/2\right),\\
u_2,&t\in\left(\mu\nu/2,\mu\nu\right),\\
u-\ve\log \frac t\mu,&t\in(\mu\nu,\mu),\\
u,&t\in\left(\mu,1\right),
\end{cases}
}
for $F^-(x;u_1,u_2)\leftarrow L(x;u_2+\ve)$ and
\eq[co14]{
\varphi_x(t)=
\begin{cases}
u-\ve\log \frac t\mu,&t\in(0,\mu),\\
u,&t\in\left(\mu,1\right),
\end{cases}
}
for $F^-(x;u_1,\infty)$. Here
\eq[co12]{
\mu=\frac{x_1-u}\ve,\quad\nu=e^{(u-u_2)/\ve},\quad\hbox{and}\quad
u=u_-\!=x_1-\ve+\sqrt{\ve^2-x_2+x_1^2}\,.
}
Lemma~\ref{LF+} and Lemma~\ref{F+} can be reformulated, respectively, as
\begin{lemma}
\label{LF-}
The function $\varphi_x$ given by~\eqref{co12.5} and~\eqref{co12} is
an optimizer for\\ $F^-(x;u_1,u_2)\longleftarrow L(x;u_2+\ve).$
\end{lemma}
\begin{lemma}
\label{F-}
The function $\varphi_x$ given by~\eqref{co14} and~\eqref{co12}
is an optimizer for $F^-(x;u_1,\infty).$
\end{lemma}
\subsection{Optimizers for sequences containing $T(x;u)$}
\label{ot}
Constructing optimizers for a $T$-type candidate is straightforward, since $T(x;u)$ 
is a linear function in $\Omega_T(u).$ Here is our strategy: for each 
point $x\in\Omega_T(u),$ draw any straight line that intersects both 
straight-line sides of $\Omega_T(u)$, but not the upper boundary of 
$\Oe,$ and then concatenate the optimizers for the two points of 
intersection in the appropriate proportion. Since $T$ is linear, 
the resulting function will be an optimizer for it. This requires 
knowing optimizers along both bounding tangents, which is consistent 
with the fact that $T$ is both left- and right-incomplete. Thus, we 
will eventually need to examine the specific sequences in which $T$ shows up in 
our global Bellman candidates, in order to write down an explicit 
optimizer for each case. However, most of the construction, as well 
as the verification of parts~(a) and~(c) of~\eqref{co1}, can be 
carried out without specifying the left and right neighbors of $T.$

Suppose $T(x;u)$ is glued to a Bellman candidate $G^-(x)$ along its 
left bounding tangent and a candidate $G^+$ along the right one. Let 
us assume that we know optimizers for $G^\pm$ along their respective 
bounding tangents.

For any $x\in\Omega_T(u)$ there is always a way to draw a line through 
$x$ intersecting the left and right tangents at the points $x^-$ and 
$x^+,$ respectively, and such that the whole segment $[x^-,x^+]$ is in 
$\Omega_T(u).$ (For example, at least one of the tangents to the upper 
boundary of $\Oe$ that pass through $x$ will satisfy this requirement. 
In fact, for each point of $\Omega_T(u)$ other than the three corners
both tangents will satisfy it.) Then, we can concatenate the two known 
optimizers, $\varphi_{x^-}$ and $\varphi_{x^+},$ to obtain a test 
function $\varphi_x$ for the point $x:$
\eq[co36]{
\varphi_x(t)=
\begin{cases}
\varphi_{x^-}\Bigl(\frac t\beta\Bigr),&t\in(0,\beta),
\\
\varphi_{x^+}\Bigl(\frac{t-\beta}{1-\beta}\Bigr),&t\in(\beta,1),
\rule{0pt}{18pt}
\end{cases}
}
where
\eq[co37]{
\beta=\frac{x_1^+-x_1}{x_1^+-x_1^-}=\frac{x_2^+-x_2}{x_2^+-x_2^-}.
}
We can verify parts~(a) and~(c) of~\eqref{co1} directly from~\eqref{co36}
and~\eqref{co37}. For the averages, we have
$$
\av{\varphi_x}{(0,1)}=x_1^-\beta+x_1^+(1-\beta)=x_1,\quad 
\av{\varphi^2_x}{(0,1)}=x_2^-\beta+x_2^+(1-\beta)=x_2.
$$
To check the optimality of $\varphi_x,$ we write 
$T=\alpha_1x_1+\alpha_2x_2+\alpha_0$ and calculate
\begin{align*}
\av{f(\varphi)}{(0,1)}&=G^-(x^-)\beta+G^+(x^+)(1-\beta)
\\
&=T(x^-;u)\beta+T(x^+;u)(1-\beta)
\\
&=(\alpha_1x_1^-+\alpha_2x_2^-+\alpha_0)\beta
+(\alpha_1x_1^++\alpha_2x_2^++\alpha_0)(1-\beta)
\\
&=\alpha_1x_1+\alpha_2x_2+\alpha_0=T(x;u).
\end{align*}
Here we have used: on the first step, the optimality of $\varphi_{x^-}$ and 
$\varphi_{x^+}$ for $G^-$ and $G^+,$ respectively; on the second and 
third  steps, the boundary conditions~\eqref{t2.015} for $T$ along the 
two bounding tangents; and on the last step, the fact that $x=\beta x^-+(1-\beta)x^+.$

Before specifying $\varphi_x$ for each sequence involving $T(x;u),$ let 
us rewrite~\eqref{co36} in the form that is independent of $x^-$ and $x^+.$ Since 
each bounding tangent is assumed to be an extremal trajectory for the 
corresponding Bellman candidate (the left tangent for $G^-,$ the right one 
for $G^+$), the candidate is linear along the tangent. Thus we can write 
$\varphi_{x^-},$ an optimizer for $G^-(x^-),$ as a concatenation of 
optimizers for $G^-(u-\ve,(u-\ve)^2+\ve^2)$ and $G^-(u,u^2)$ and 
similarly for $\varphi_{x^+}:$
\eq[co38]{
\varphi_{x^-}(t)=
\begin{cases}
\psi_{u-\ve}\Bigl(\frac t{\alpha_-}\Bigr),&t\in(0,\alpha_-),\\
u,&t\in(\alpha_-,1),
\end{cases}
\qquad
\varphi_{x^+}(t)=
\begin{cases}
u,&t\in(0,\alpha_+),\\
\psi_{u+\ve}\Bigl(\frac{t-\alpha_+}{1-\alpha_+}\Bigr),&t\in(\alpha_+,1).
\end{cases}
}
where
$$
\alpha_-=\frac{u-x_1^-}\ve,\qquad \alpha_+=1-\frac{x_1^+-u}\ve.
$$
Observe that in the expression for $\varphi_{x^-}$ we put the constant 
value $u$ on the right part of $(0,1),$ while in the expression for 
$\varphi_{x^+}$ this constant value is on the left. This is done so 
as to minimize the $\BMO$ norm of the the resulting function 
$\varphi_x$ given by~\eqref{co36} (see Remark~\ref{care}). 
Using~\eqref{co36} in conjunction with~\eqref{co38}, we get
\eq[co41]{
\varphi_x(t)=
\begin{cases}
\psi_{u-\ve}\Bigl(\frac t{\beta\alpha_-}\Bigr),&t\in(0,\beta\alpha_-),
\\
u,&t\in(\beta\alpha_-,\beta+(1-\beta)\alpha_+),
\\
\psi_{u+\ve}\Bigl(\frac{t-\beta-(1-\beta)\alpha_+}{(1-\beta)(1-\alpha_+)}
\Bigr),&t\in(\beta+(1-\beta)\alpha_+,1).
\end{cases}
}
Since $x$ is a unique --- and independent of $x^-$ and $x^+$ --- 
convex combination of the points $(u-\ve,u^2-2\ve u+2\ve^2),$ 
$(u,u^2),$ and $(u+\ve,u^2+2\ve u+2\ve^2),$ we expect the weights 
$\beta\alpha_-$ and $(1-\beta)(1-\alpha_+)$ in~\eqref{co41} not to 
depend on how the points $x^\pm$ were chosen. Indeed, after a bit 
of algebra, we can rewrite~\eqref{co41} as
\eq[co42]{
\varphi_x(t)=
\begin{cases}
\psi_{u-\ve}\Bigl(\frac t{\mu_-}\Bigr),&t\in(0,\mu_-),
\\
u,&t\in(\mu_-,1-\mu_+),
\\
\psi_{u+\ve}\Bigl(\frac{t-1+\mu^+}{\mu_+}\Bigr),&t\in(1-\mu_+,1),
\end{cases}
}
where
\eq[co43]{
\mu_-=\frac{x_2-2ux_1+u^2}{4\ve^2}-\frac{x_1-u}{2\ve},\qquad 
\mu_+=\frac{x_2-2ux_1+u^2}{4\ve^2}+\frac{x_1-u}{2\ve}.
}

We now need to check that $\varphi_x$ constructed according 
to~\eqref{co42} will be in $\BMO_\ve$ for each global context of $T.$ 
Among our four canonical blocks, a $T$ block can only have $F^+$ or 
$L$ blocks glued to its left and only $F^-$ or, again, $L$ blocks to 
its right. Since we have already constructed optimizers for all $F^\pm$ 
and $L$ blocks, we could proceed in generality and prove that an 
optimizer of the form~\eqref{co42} could always be rearranged --- 
separately on $(0,\mu_-)$ and $(1-\mu_+,1)$ --- so that the resulting 
function is in $\BMO_\ve.$ However, we choose here to be more explicit 
and consider specific optimizers for the specific sequences in which 
$T$ appears. We have three such sequences an so split further 
presentation in three parts.
\subsubsection{Optimizers for $F^+(x;-\infty,0)\!\!\longrightarrow\! T_0(x)
\!\longleftarrow\!F^-(x;0,\infty)$}

This sequence appears in~\eqref{m2.9}. For this case~\eqref{co42} gives
\eq[co42.1]{
\varphi_x(t)=
\begin{cases}
\ve\log\Bigl(\frac t{\mu_-}\Bigr),&t\in(0,\mu_-),
\\
0,&t\in(\mu_-,1-\mu_+),
\\
-\ve\log\Bigl(\frac{t-1+\mu^+}{\mu_+}\Bigr),&t\in(1-\mu_+,1).
\end{cases}
}
To show that $\varphi_x\in\BMO_\ve,$ take an interval 
$(c,d)\subset(0,1).$ First, observe that if $c\in[\mu_-,1-\mu_+],$ the 
Bellman point $x^{(c,d)}$ is the same as the one for the cut-off at 
height $0$ of the function $\ve\log(t/\mu_-),$ which has been shown to 
be in $\BMO_\ve;$ therefore, $x^{(c,d)}$ is in $\Oe.$ The same reasoning 
applies when $d\in[\mu_-,1-\mu_+].$ Therefore, we will assume that 
$c\in(0,\mu_-)$ and $d\in(1-\mu_+,1).$

Recall representation~\eqref{co37}, which gives $x$ as a convex combination 
of $x^-$ and $x^+.$ In our standard notation, $x^-=x^{(0,\beta)}$ and 
$x^+=x^{(\beta,1)}.$ We first would like to determine the location of 
the point $x^{(c,\beta)}.$ We know that $x^{(0,c)}$ is the Bellman point 
for the interval $(0,c)$ and the logarithm $\ve\log(t/\mu_-)$ and we have 
already seen that every such point is on the parabola $x_2=x_1^2+\ve^2.$ 
Therefore, $x^{(0,c)}$ is above the line $m$ through $x^-$ and $x^+$ 
(see Figure~\ref{fig166}).
\begin{figure}[ht]
  \centering{\includegraphics[width=12cm]{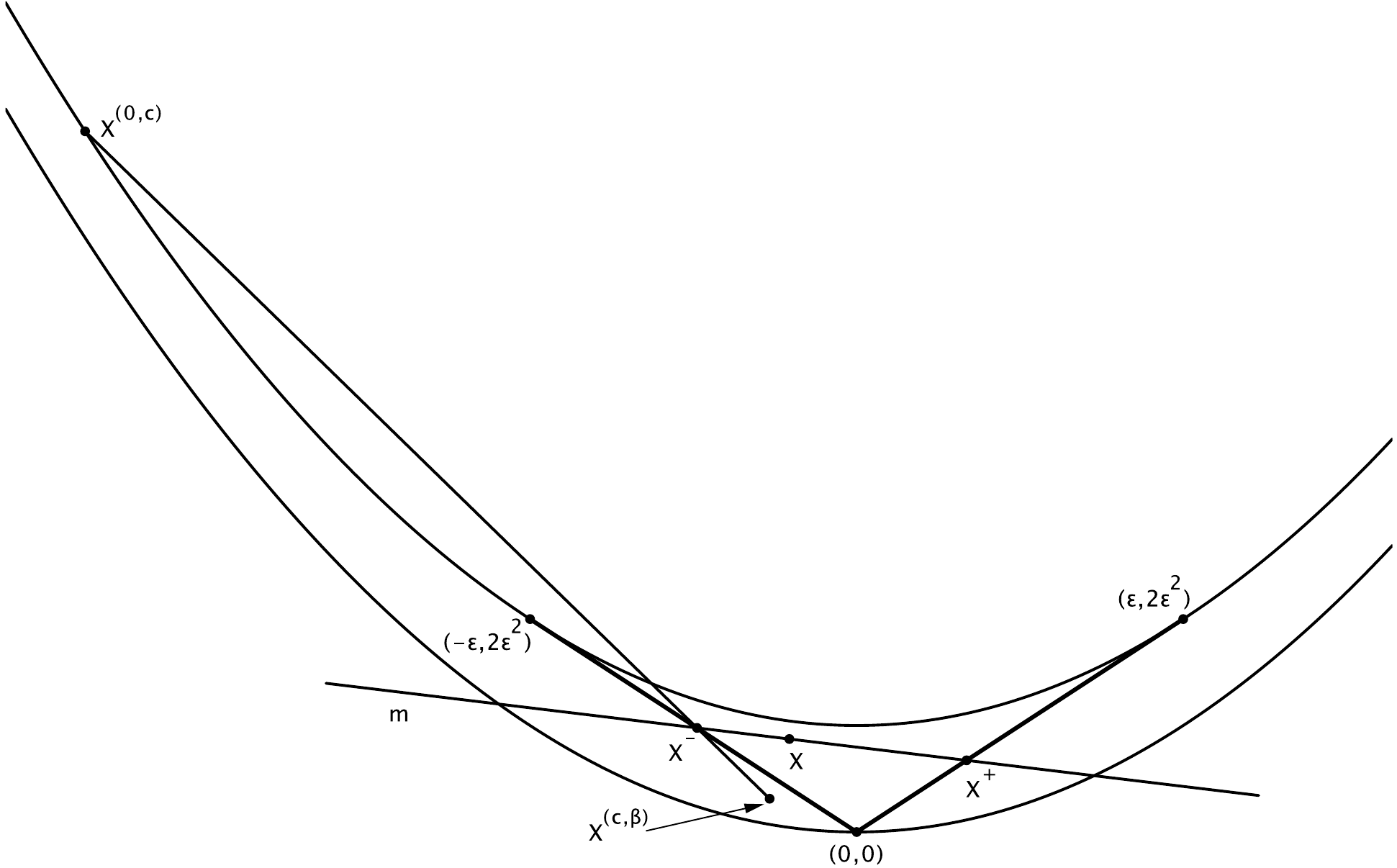}}
\caption{The location of the point $x^{(c,\beta)}$}
\label{fig166}
\end{figure}
Since $x^-$ is a convex combination of $x^{(0,c)}$ and $x^{(c,\beta)},$ 
we conclude that $x^{(c,\beta)}$ is below $m.$ Similarly, the point 
$x^{(\beta,d)}$ (not shown in the figure) and, therefore, the whole 
line segment $[ x^{(c,\beta)},x^{(\beta,d)}]$ is below $m$ and so in 
$\Oe.$ The point $x^{(c,d)}$ is a convex combination of $x^{(c,\beta)}$ 
and $x^{(\beta,d)},$ which means that it is on the line segment and, 
thus, in $\Oe.$
\subsubsection{Optimizers for $L_0(x)\longrightarrow F^+(x;\ve,\xi)
\longrightarrow T(x;\xi)\longleftarrow F^-(x;\xi,\infty)$}
This sequence appears in~\eqref{h11}. Note that we have to include in 
the sequence the block $L_0$ to the left of $F^+,$ since $F^+$ is 
left-incomplete. For this case~\eqref{co42} yields
\eq[co42.2]{
\varphi_x(t)=
\begin{cases}
-\ve,&t\in(0,\mu_-\nu/2),\\
\ve,&t\in(\mu_-\nu/2,\mu_-\nu),\rule{0pt}{12pt}\\
\xi+\ve\log\Bigl(\frac t{\mu_-}\Bigr),&t\in(\mu_-\nu,\mu_-),\\
\xi,&t\in(\mu_-,1-\mu_+),\\
\xi-\ve\log\Bigl(\frac{t-1+\mu^+}{\mu_+}\Bigr),&t\in(1-\mu_+,1),
\end{cases}
}
where $\nu=e^{1-\xi/\ve}$ and $\mu_\pm$ are given by~\eqref{co43}.

Instead of $\varphi_x,$ it is convenient to consider $\tilde{\varphi}_x
=\varphi_x-\xi:$
$$
\tilde{\varphi}_x(t)=
\begin{cases}
-\xi-\ve,&t\in(0,\mu_-\nu/2),\\
-\xi+\ve,&t\in(\mu_-\nu/2,\mu_-\nu),\rule{0pt}{12pt}\\
\ve\log\Bigl(\frac t{\mu_-}\Bigr),&t\in(\mu_-\nu,\mu_-),\\
0,&t\in(\mu_-,1-\mu_+),\\
-\ve\log\Bigl(\frac{t-1+\mu^+}{\mu_+}\Bigr),&t\in(1-\mu_+,1).
\end{cases}
$$
Clearly, $\varphi_x\in\BMO_\ve\Longleftrightarrow 
\tilde{\varphi}_x\in\BMO_\ve.$ Let us take a subinterval $(c,d)$ of 
$(0,1).$ Then $\tilde{\varphi}_x|_{(\mu_-\nu,1)}$ is just (the 
appropriate restriction of) the optimizer from the previous subsection, 
given by~\eqref{co42.1}, and so in $\BMO_\ve.$ Furthermore, 
$\tilde{\varphi}_x|_{(\mu_-\nu/2,1)}$ is the cut-off of that optimizer 
at height $\ve-\xi,$ restricted to $(\mu_-\nu/2,1),$ and. hence, it is also
$\BMO_\ve.$ Therefore, in proving that the Bellman point $x^{(c,d)}$ for 
$\tilde{\varphi}_x$ is in $\Oe$ we only need to consider 
$0<c<\mu_-\nu/2.$ Where should we place $d?$ If $d\le1-\mu_+,$ 
$x^{(c,d)}$ can be seen to be a Bellman point for the optimizer for 
the sequence $L\to F^+,$ given by \eqref{co10.5} with $u_1=-\xi+\ve$ 
and $u=0;$ therefore, it is in $\Oe.$ Thus, from now on we will assume 
that $c<\mu_-\nu/2<1-\mu_+<d.$

As before, we appeal to representation~\eqref{co36}--\eqref{co37}. The 
only difference now is that the point $x^{(0,c)}=(-\xi-\ve,(\xi+\ve)^2)$ 
will not be on the upper parabola $x_2=x_1^2+\ve^2,$ but instead on the 
lower parabola $x_2=x_1^2.$ Since $\xi>\ve,$ this point is above the 
line $x_2=-2\ve x_1.$ Therefore, it is above the line through $x^-$ and 
$x^+$ (let us again call it $m$). Since $x^-=x^{(0,\beta)}$ is a convex 
combination of $x^{(0,c)}$ and $x^{(c,\beta)},$ we conclude that 
$x^{(c,\beta)}$ is below $m.$ From here the consideration is identical 
to the one in the previous subsection and we conclude that 
$x^{(c,d)}\in\Oe.$
\subsubsection{Optimizers for $L^+(x;-\ve)\longrightarrow T_0(x)
\longleftarrow L^-(x;\ve)$}
This sequence appears in~\eqref{h27}. From~\eqref{co42} we obtain
\eq[co42.3]{
\varphi_x(t)=
\begin{cases}
-2\ve,&t\in(0,\mu_-/2),\\
0,&t\in(\mu_-/2,1-\mu_+/2),\\
2\ve,&t\in(1-\mu_+/2,1).
\end{cases}
}
To prove that $\varphi_x\in\BMO_\ve,$ take an interval 
$(c,d)\subset(0,1).$ If $c\ge\mu_-/2$ and/or $d\le 1-\mu_+/2,$ the 
Bellman point $x^{(c,d)}$ is a convex combination of one of 
$(\pm2\ve,4\ve^2)$ and $(0,0).$ The point $x^{(c,d)}$ then is on one 
of the tangent lines $x_2=\pm2\ve x_1$ and, thus, in $\Oe.$ Let us assume 
that $c<\mu_-/2$ and $1-\mu_+/2<d.$ More fully, let us write 
$c<\mu_-/2<\beta<1-\mu_+/2<d,$ where $\beta$ is given by~\eqref{co37}. 
We have $x^-=x^{(0,\beta)}$ as a convex combination of $x^{(0,c)}=
(-2\ve,4\ve^2)$ and $x^{(c,\beta)}.$ Therefore, $x^{(c,\beta)}$ is on 
the tangent $x_2=-2\ve x_1,$ below $x^-.$ Similarly, $x^{(\beta,d)}$ 
is below $x^+.$ We conclude that $x^{(c,d)}\in\Oe.$

We have completed the proofs of the following sequence of lemmas.
\begin{lemma}
\label{ot1.1}
The function~\eqref{co42.1}\textup, with $\mu_\pm$ given by~\eqref{co43} 
for $u=0,$ is an optimizer for $F^+(x;-\infty,0)\longrightarrow T_0(x)
\longleftarrow F^-(x;0,\infty).$
\end{lemma}
\begin{lemma}
\label{ot1.2}
The function~\eqref{co42.2}\textup, with $\nu=e^{1-\xi/\ve}$ and 
$\mu_\pm$ given by~\eqref{co43} for $u=\xi,$ is an optimizer for 
$L_0(x)\longrightarrow F^+(x;\ve,\xi)\longrightarrow T(x;\xi)
\longleftarrow F^-(x;\xi,\infty).$
\end{lemma}
\begin{lemma}
\label{ot1.3}
The function~\eqref{co42.3}\textup, with $\mu_\pm$ given by~\eqref{co43} 
for $u=0,$ is an optimizer for $L^+(x;-\ve)\longrightarrow T_0(x)
\longleftarrow L^-(x;\ve).$
\end{lemma}

Putting together Lemmas~\ref{oL0}--\ref{ot1.3}, we obtain
\begin{theorem}
\label{imp2}
For any $x\in\Oe,$ we have
$$
\begin{array}{lll}
For~p\ge2:&N_{\ve,p}(x)\le\bel{B}_{\ve,p}(x),
&M_{\ve,p}(x)\ge\bel{b}_{\ve,p}(x),
\\
For~1\le p< 2:&M_{\ve,p}(x)\le\bel{B}_{\ve,p}(x), 
&N_{\ve,p}(x)\ge\bel{b}_{\ve,p}(x),
\\
For~0<p<1:&P_{\ve,p}(x)\le\bel{B}_{\ve,p}(x), 
&R_{\ve,p}(x)\ge\bel{b}_{\ve,p}(x).
\end{array}
$$
\end{theorem}
\section{Proofs of the main inequalities}
\label{proofs}
Theorems~\ref{imp1} and~\ref{imp2} give us the explicit expressions 
for $\bel{B}_{\ve,p}(x)$ and $\bel{b}_{\ve,p}(x)$ for all $p>0:$
\begin{theorem}
\label{imp3}
For any $x\in\Oe,$ we have
$$
\begin{array}{lll}
For~p\ge2:&\bel{B}_{\ve,p}(x)=N_{\ve,p}(x),
&\bel{b}_{\ve,p}(x)=M_{\ve,p}(x),
\\
For~1\le p< 2:&\bel{B}_{\ve,p}(x)=M_{\ve,p}(x),
&\bel{b}_{\ve,p}(x)=N_{\ve,p}(x),
\\
For~0<p<1:&\bel{B}_{\ve,p}(x)=P_{\ve,p}(x),
&\bel{b}_{\ve,p}(x)=R_{\ve,p}(x).
\end{array}
$$
\end{theorem}
We are now in a position to prove all the theorems stated in 
Section~\ref{inequalities}.
\begin{proof}[Proof of Theorem~\ref{th1.0}]
It suffices to consider $Q=(0,1).$ Take any $\varphi\in\BMO(Q)$ and let $\ve=\|\varphi\|_{\BMO(Q)},$ $x_1=\av{\varphi}Q,$ $x_2=
\av{\varphi^2}Q.$ Then
$$
\bel{b}_{\ve,p}(x_1,x_2)\le\av{|\varphi|^p}Q
\le\bel{B}_{\ve,p}(x_1,x_2).
$$
Replacing $\varphi$ with $\varphi-\av{\varphi}Q$ gives
\eq[ee1]{
\bel{b}_{\ve,p}(0,x_2-x_1^2)\le\av{|\varphi-\av{\varphi}Q|^p}Q
\le\bel{B}_{\ve,p}(0,x_2-x_1^2).
}
We now invoke Theorem~\ref{imp3}, for which we need the exact expressions for candidates $M,N,P,$ and $R.$ They come, respectively, from~\eqref{ff6}, \eqref{m3}, \eqref{h12}, and~\eqref{h28}:
$$
M_{\ve,p}(0,x_2)=x_2^{p/2};~ 
N_{\ve,p}(0,x_2)=\frac p2\Gamma(p)\ve^{p-2}x_2;~ 
P_{\ve,p}(0,x_2)=x_2^{p/2};~
R_{\ve,p}(0,x_2)=2^{p-2}\ve^{p-2}x_2.
$$
Plugging these into~\eqref{ee1} yields the stated inequalities. Furthermore, these inequalities are sharp because each becomes an equality, 
if we take $\varphi$ to be the corresponding optimizer 
$\varphi_{(0,\ve^2)}$ from Section~\ref{converse}.  Specifically, let
\eq[v1]{
\varphi_1=
\begin{cases}
-\ve;&\!\!\!t\in\big(0,\frac12\big),\\
\ve;&\!\!\!t\in\big(\frac12,1\big),
\end{cases}
\quad
\varphi_2=
\begin{cases}
\ve\log(4t),&\!\!\!t\in\big(0,\frac14\big),\\
0,&\!\!\!t\in\big(\frac14,\frac34\big),\\
-\ve\log(4-4t),&\!\!\!t\in\big(\frac34,1\big),
\end{cases}
\quad
\varphi_3=
\begin{cases}
-2\ve,&\!\!\!t\in\big(0,\frac18\big),\\
0,&\!\!\!t\in\big(\frac18,\frac78\big),\\
2\ve,&\!\!\!t\in\big(\frac78,1\big).
\end{cases}
}
These are the optimizers given by~\eqref{ol01}, \eqref{co42.1}, 
and~\eqref{co42.3}, respectively, each constructed for the point 
$(0,\ve^2).$ As was shown in Section~\ref{converse},
$\|\varphi_k\|_{\BMO(Q)}=(\av{\varphi^2}Q-\av{\varphi}Q^2)^{1/2}=\ve,$ $k=1,2,3.$ On the other hand,
$$
\av{|\varphi_1-\av{\varphi_1}Q|^p}Q=\av{|\varphi_1|^p}Q=\ve^p,
$$
$$
\av{|\varphi_2-\av{\varphi_2}Q|^p}Q=\av{|\varphi_2|^p}Q=
2\ve^p\cdot\frac14\int_0^1\!\!|\log t|^p\,dt=\frac12\Gamma(p+1)\ve^p,
$$
and
$$
\av{|\varphi_3-\av{\varphi_3}Q|^p}Q=\av{|\varphi_3|^p}Q
=2\cdot\frac18\cdot(2\ve)^p=2^{p-2}\ve^p.
$$
\end{proof}
\begin{proof}[Proof of Theorem~\ref{th1.1}]
Again, set $Q=(0,1).$ Take any $\varphi\in\BMO(Q)$ and let $\ve=\|\varphi\|_{\BMO(Q)}.$ For 
any subinterval $J$ of $Q$ let $x^J_1=\av{\varphi}J,$ $x^J_2=
\av{\varphi^2}J.$ Arguing as in the previous proof, we have
$$
\bel{b}_{\ve,p}(0,x^J_2-(x^J_1)^2)\le\av{|\varphi-\av{\varphi}J|^p}J
\le\bel{B}_{\ve,p}(0,x^J_2-(x^J_1)^2).
$$

Let us consider these inequalities separately. For the one on the right, 
using that $\bel{B}_{\ve,p}(0,\cdot)$ is increasing, we have
$$
\av{|\varphi-\av{\varphi}J|^p}J\le\bel{B}_{\ve,p}(0,\ve^2).
$$
Taking the supremum over all $J,$ we get
$$
\|\varphi\|^p_{\BMO^p(Q)}\le\bel{B}_{\ve,p}(0,\ve^2).
$$
For each $p$ this inequality is sharp: for $p<2$ it is attained for $\varphi=\varphi_1$ from~\eqref{v1}, and for $p>2$ it is attained for $\varphi=\varphi_2.$ 

For the inequality on the left, take a sequence $\{J_n\}$ of 
subintervals of $Q$ such that
$$
\lim_{n\to\infty}\big(x^{J_n}_2-(x_1^{J_n})^2\big)=\ve^2.
$$
Using the continuity of $\bel{b}_{\ve,p}(0,\cdot),$ we have
$$
\bel{b}_{\ve,p}(0,\ve^2)
=\lim_{n\to\infty}\bel{b}_{\ve,p}\big(0,x^{J_n}_2-(x_1^{J_n})^2\big)
\le\limsup_{n\to\infty}\av{|\varphi-\av{\varphi}{J_n}|^p}{J_n}
\le\|\varphi\|^p_{\BMO^p(Q)}.
$$
This inequality is sharp for $p>2:$ it is attained, again, for $\varphi=\varphi_1$ from~\eqref{v1}.
\end{proof}
\begin{remark}
The last calculation in the proof shows why $\varphi_2$ and $\varphi_3$ cannot be used to show sharpness of the norm estimates in the cases $1\le p<2$ and $0<p\le 1,$ respectively. 
Consider $p=1.$ The issue is that, while each function attains its $BMO^2$ norm on $(0,1),$ neither attains its $\BMO^1$ norm on this interval. Indeed, one can easily calculate that both functions have $1$-oscillations equal to $\ve/2$ on $(0,1).$
However,
$$
\|\varphi_2\|_{\BMO^1}\ge\av{|\varphi_2-\av{\varphi_2}{(0,1/4)}|}{(0,1/4)}=\frac{2\ve}e,\quad \|\varphi_3\|_{\BMO^1}\ge\av{|\varphi_3-\av{\varphi_3}{(0,1/4)}|}{(0,1/4)}=\ve.
$$
\end{remark}
\begin{proof}[Proof of Theorem~\ref{th2.1}]
We can set $Q=(0,1).$ If $\|\varphi\|_{\BMO(Q)}=0,$ there is nothing to prove. Assuming this is not the case, we use
Theorem~\ref{th1.0} to get
$$
\frac{2\|\varphi\|^{2-p_2}_{\BMO(Q)}}{p_2\Gamma(p_2)}\av{|\varphi-\av{\varphi}Q|^{p_2}}Q\le
\av{\varphi^2}Q-\av{\varphi}Q^2\le
\frac{2\|\varphi\|^{2-p_1}_{\BMO(Q)}}{p_1\Gamma(p_1)}\av{|\varphi-\av{\varphi}Q|^{p_1}}Q.
$$
Each of these inequalities becomes an equality when $\varphi=\varphi_2$ 
from \eqref{v1}. The left-hand side inequality in the statement of the theorem is attained, for instance, for $\varphi=\varphi_1.$
\end{proof}
\begin{proof}[Proof of Theorem~\ref{th3}]
Set $Q=(0,1);$ all averages will be over $Q.$ Let 
$\ve=\|\varphi\|_{\BMO(Q)}$ and assume $\ve<1.$ We have
\begin{align*}
\ave{e^{|\varphi-\ave{\varphi}|}}
&=\sum_{k=0}^\infty\frac1{k!}\ave{|\varphi-\ave{\varphi}|^k}
\le1+\bel{B}_{\ve,1}(0,\ve^2)
+\sum_{k=2}^\infty\frac1{k!}\bel{B}_{\ve,k}(0,\ve^2)
\\
&=1+M_{\ve,1}(0,\ve^2)+\sum_{k=2}^\infty\frac1{k!}N_{\ve,k}(0,\ve^2)
=1+\ve+\frac12\sum_{k=2}^\infty\ve^k=\frac{1-\frac{\ve^2}2}{1-\ve}.
\end{align*}

On the other hand, taking $\varphi=\varphi_2$ from \eqref{v1}, we get
$$
\ave{e^{|\varphi_2-\ave{\varphi_2}|}}
=2\cdot\frac14\int_0^1\!\!e^{\ve|\log t|}\,dt+2\cdot\frac14
=\frac{1-\frac\ve2}{1-\ve}.
$$
This calculation also shows that the bound $\ve_0=1$ is sharp.
\end{proof}
\section{Other choices of $f$}
\label{other}
Throughout the paper, we have concentrated on one specific boundary 
function $f:$ $f(s)=|s|^p,$ $p>0.$ However, the machinery developed
in these pages works for many other choices of $f.$ Let us briefly 
describe several such choices and their \Bfs\ without going into details.
\subsection{$f(s)=\log|s|$}
As mentioned earlier, this function corresponds to the case $p=0,$ since
$$
\lim_{p\to0}\ave{|\varphi|^p}^{1/p}=e^{\ave{\log|\varphi|}}.
$$
It is easy to show that the corresponding \Bfs\ are
\eq[01]{
\bel{B}_{\ve,0}(x)=P_{\ve,0}(x)\quad\text{and}\quad \bel{b}_{\ve,0}(x)=-\infty.
}
Here $P_{\ve,0}$ is given by \eqref{h11}, with every block re-specified 
for $f(s)=\log|s|.$ Thus, $F^+$ and $F^-$ are given by~\eqref{gg4.0} 
and~\eqref{gg4.01}, respectively; $T$ is given by~\eqref{t4}; and $L_0,$ 
given by~\eqref{g1}, is simply $\frac12\log x_2.$ To show that this is 
a viable global candidate, one needs Lemma~\ref{lm}. To prove the 
statement for $\bel{B},$ use the local concavity of the candidate 
$P_{\ve,0}$ to run the induction of Section~\ref{induction} and then 
apply the optimizer for $P_{\ve,p}$ from Section~\ref{converse}. To 
prove the statement for $\bel{b},$ simply use the optimizer for 
$R_{\ve,p}$ from Section~\ref{converse}.

From \eqref{01}, we have sharp inequalities for $\varphi\in\BMO_\ve:$
\eq[01.5]{
-\infty\le\ave{\log|\varphi|}\le P_{\ve,0}(\ave{\varphi},\ave{\varphi}^2),
}
and so
\eq[01.6]{
0\le e^{\ave{\log|\varphi-\ave{\varphi}|}}\le 
e^{P_{\ve,0}(0,\ave{\varphi^2}-\ave{\varphi}^2)}
=e^{\frac12\log(\ave{\varphi^2}-\ave{\varphi}^2)}\le\ve.
}
While the second inequality in~\eqref{01.5} is non-trivial and gives 
a sharp estimate on $\ave{\log|\varphi|}$ for any pair of specified 
averages of $\varphi,$ its immediate consequence, the second-from-left 
inequality in~\eqref{01.6}, is not interesting, as it simply expresses 
the norm monotonicity~\eqref{d0.2}. However, the leftmost inequality 
in~\eqref{01.6} is important, as its sharpness means that 
$\BMO\subsetneq\BMO^0.$ It is also the same result as one gets from 
the top line in Theorem~\ref{th1.1}, by taking the limit as $p\to0^+.$
\subsection{$f(s)=|s|^p,$ $p<0$}
For this case, the situation reverses, compared to the previous one. 
We have
$$
\bel{B}_{\ve,p}(x)=\infty,\qquad \bel{b}_{\ve,p}(x)=P_{\ve,p}(x),
$$
with $P_{\ve,p}$ given by~\eqref{h11} and so~\eqref{h12}, and so
$$
P_{\ve,p}(\ave{\varphi},\ave{\varphi}^2)\le\ave{|\varphi|^p}\le\infty,
$$
which produces the sharp inequalities
$$
0\le\ave{|\varphi-\ave{\varphi}|^p}^{1/p}
\le\big[P_{\ve,p}(0,\ave{\varphi}^2-\ave{\varphi}^2)\big]^{1/p}
=(\ave{\varphi}^2-\ave{\varphi}^2)^{1/2}\le\ve.
$$
\subsection{$f(s)=e^{|s|}-|s|,$ $\ve<1$} This is the function that 
implicitly allowed us to prove the John--Nirenberg estimates of 
Theorem~\ref{th3}. As we saw in Section~\ref{proofs}, the key fact 
is that all upper \Bfs\ $\bel{B}_{\ve,p}$ and, separately, all 
lower \Bfs\ $\bel{b}_{\ve,p}$ have identical optimizers for $p\ge2.$ 
Since the Taylor expansion for $e^{|s|}-|s|$ is missing the term 
corresponding to $p=1,$ we expect the Bellman foliation of $\Oe$ to be 
the same as for $|s|^p,$ $p\ge2$. Indeed, we can easily show that
\eq[03]{
\bel{B}_{\ve,f}(x)=N_{\ve,f}(x),\qquad \bel{b}_{\ve,f}(x)=M_{\ve,f}(x),
}
where $M_{\ve,f}$ and $N_{\ve,f}$ are given by~\eqref{ff5.9} 
and~\eqref{m2.9}, with their blocks re-specified for this choice of 
$f.$ Therefore, after a small bit of calculation, we have the sharp 
inequalities
$$
e^{\sqrt{\ave{\varphi^2}-\ave{\varphi}^2}}
-\sqrt{\ave{\varphi^2}-\ave{\varphi}^2}
\le\ave{e^{|\varphi-\ave{\varphi}|}}-\ave{|\varphi-\ave{\varphi}|}
\le\frac{\ave{\varphi^2}-\ave{\varphi}^2}{2(1-\ve)}+1
\le\frac{1-\ve+\frac{\ve^2}2}{1-\ve}.
$$
In fact, if $f$ is any linear combination of powers greater than or 
equal to $2,$ with non-negative coefficients, the \Bfs\ will be given 
by~\eqref{03}. If, on the other hand, $f$ is a linear combination of 
powers between $1$ and $2$ with non-negative coefficients, the upper 
and lower \Bfs\ will switch, i.e. they will be given by $M_{\ve,f}$ 
and $N_{\ve,f},$ respectively. Further generalizations along these 
lines are possible.

\end{document}